\documentclass[11pt,a4paper,oneside,reqno]{amsart}
\usepackage{amsmath,amsfonts,amssymb,amsthm,fancyhdr,bm,mathtools,enumitem}
\usepackage[hidelinks]{hyperref}
\usepackage{fullpage}
\usepackage[dvipsnames]{xcolor}
\usepackage[capitalize]{cleveref}
\usepackage[color=Purple!50!white]{todonotes}
\usepackage[numbers,sort]{natbib}
\usepackage{tikz,ifthen}

\usetikzlibrary{decorations.markings,arrows.meta}
\tikzset{vert/.style={draw, fill=black, circle, inner sep=2pt}}

\newtheorem{theorem}{Theorem}[section]
\newtheorem{lemma}[theorem]{Lemma}
\newtheorem{proposition}[theorem]{Proposition}

\newtheorem{conjecture}[theorem]{Conjecture}
\newtheorem{claim}[theorem]{Claim}
\theoremstyle{definition}
\newtheorem{definition}[theorem]{Definition}
\newtheorem*{remark}{Remark}

\crefname{equation}{equation}{equations}
\crefname{lemma}{Lemma}{Lemmas}
\crefname{proposition}{Proposition}{Propositions}
\crefname{claim}{Claim}{Claims}
\crefname{theorem}{Theorem}{Theorems}
\crefname{conjecture}{Conjecture}{Conjectures}
\crefname{figure}{Figure}{Figures}

\newlist{lemenum}{enumerate}{1}
\setlist[lemenum]{label=(\alph*), ref=\thelemma(\alph*)}
\crefalias{lemenumi}{lemma} 

\DeclareMathOperator\pr{Pr}

\newcommand\ab[1]{\lvert#1\rvert}

\newcommand{\flo}[1]{\lfloor #1 \rfloor}
\newcommand{\cei}[1]{\lceil #1 \rceil}
\newcommand{\cop}[1]{\widehat{#1}}

\let\oldH\H
\let\oldL\L

\let\leq\leqslant
\let\geq\geqslant

\newcommand{\Z}{\mathbb{Z}}

\newcommand\I{\mathcal{I}}
\newcommand\F{\mathcal{F}}
\newcommand\G{\mathcal{G}}
\renewcommand\H{\mathcal{H}}
\renewcommand\L{\mathcal{L}}
\renewcommand\S{\mathcal{S}}

\newcommand\Gbad{\mathcal{G}_{\mathrm{bad}}}

\usepackage{stmaryrd}
\newcommand{\br}[1]{\llbracket{#1}\rrbracket}

\title{On the Kohayakawa--Kreuter conjecture}
\author{Eden Kuperwasser}
\address{School of Mathematical Sciences, Tel Aviv University, Tel Aviv 6997801, Israel}
\email{{\char '173}kuperwasser,samotij,yuvalwig{\char '175}@tauex.tau.ac.il}

\author{Wojciech Samotij}

\author{Yuval Wigderson}

\thanks{EK, WS, and YW are supported by ERC Consolidator Grant 101044123 (RandomHypGra), by Israel Science Foundation Grant 2110/22, and by NSF--BSF Grant 2019679. YW is additionally supported by ERC Consolidator Grant 863438 (LocalGlobal).}
\date{}
\begin{document}

\begin{abstract}
	Let us say that a graph $G$ is Ramsey for a tuple $(H_1,\dots,H_r)$ of graphs if every $r$-coloring of the edges of $G$ contains a monochromatic copy of $H_i$ in color $i$, for some $i \in \br r$.
	A famous conjecture of Kohayakawa and Kreuter, extending seminal work of R\"odl and Ruci\'nski, predicts the threshold at which the binomial random graph $G_{n,p}$ becomes Ramsey for $(H_1,\dots,H_r)$ asymptotically almost surely. In this paper, we resolve the Kohayakawa--Kreuter conjecture for almost all tuples of graphs. Moreover, we reduce its validity to the truth of a certain deterministic statement, which is a clear necessary condition for the conjecture to hold. All of our results actually hold in greater generality, when one replaces the graphs $H_1,\dots,H_r$ by finite families $\H_1,\dots,\H_r$. 
	Additionally, we pose a natural (deterministic) graph-partitioning conjecture, which we believe to be of independent interest, and whose resolution would imply the Kohayakawa--Kreuter conjecture.
\end{abstract}
\maketitle

\section{Introduction}
\subsection{Symmetric Ramsey properties of random graphs}

Given graphs $G$ and $H_1, \dotsc, H_r$, one says that $G$ is \emph{Ramsey for the tuple $(H_1,\dotsc,H_r)$} if, for every $r$-coloring of the edges of $G$, there is a monochromatic copy of $H_i$ in some color $i \in \br{r}$.  In the symmetric case $H_1=\dotsb=H_r=H$, we simply say that $G$ is \emph{Ramsey for $H$ in $r$ colors}.  Ramsey's theorem~\cite{MR1576401} implies that the complete graph $K_n$ is Ramsey for $(H_1,\dotsc,H_r)$ whenever $n$ is sufficiently large.  The fundamental question of graph Ramsey theory is to determine, for a given tuple $(H_1, \dotsc, H_r)$, which graphs $G$ are Ramsey for it.  For more on this question, as well as the many fascinating sub-questions it contains, we refer the reader to the survey~\cite{MR3497267}.

In this paper, we are interested in Ramsey properties of random graphs, a topic that was initiated in the late 1980s by Frankl--R\"odl~\cite{MR0932121} and \oldL{}uczak--Ruci\'nski--Voigt~\cite{MR1182457}. The main question in this area is, for a given tuple $(H_1,\dotsc,H_r)$, which functions $p=p(n)$ satisfy that $G_{n,p}$ is Ramsey for $(H_1,\dots,H_r)$ a.a.s.\footnote{As usual, $G_{n,p}$ denotes the binomial random graph with edge probability $p$ and we say that an event happens \emph{asymptotically almost surely (a.a.s.)} if its probability tends to $1$ as $n \to \infty$.}  In the case $H_1=\dotsb=H_r$, this question was resolved in the remarkable work of R\"odl and Ruci\'nski~\cite{MR1249720,MR1262978,MR1276825}. In order to state their result, we need the following terminology and notation. For a graph $J$, we denote by $v_J$ and $e_J$ the number of vertices and edges, respectively, of $J$. 
The \emph{maximal $2$-density} of a non-empty graph $H$ with $v_H \geq 3$ is then defined\footnote{We also define $m_2(K_2) \coloneqq 1/2$ and $m_2(H) \coloneqq 0$ if $H$ has no edges.} to be
\[
	m_2(H) \coloneqq \max \left\{ \frac{e_J-1}{v_J-2} : J \subseteq H, v_J \geq 3\right\}.
\]
With this notation, we can state the random Ramsey theorem of R\"odl and Ruci\'nski~\cite{MR1276825}.
\begin{theorem}[R\"odl--Ruci\'nski \cite{MR1276825}]
  \label{thm:RR}
  For every graph $H$ which is not a forest\footnote{R\"odl and Ruci\'nski also determined the Ramsey threshold when $H$ is a forest, but for simplicity we do not state this more general result.} and every integer $r \geq 2$, there exist constants $c,C>0$ such that
  \[
    \lim_{n \to \infty}\pr(G_{n,p}\text{ is Ramsey for $H$ in $r$ colors}) = 
    \begin{cases}
      1&\text{if } p \geq Cn^{-1/m_2(H)},\\
      0&\text{if } p \leq cn^{-1/m_2(H)}.
    \end{cases}
  \]
\end{theorem}
As with many such threshold results for random graph properties, \cref{thm:RR} really consists of two statements: the \emph{$1$-statement}, which says that $G_{n,p}$ satisfies the desired property a.a.s.\ once $p$ is above some threshold, and the \emph{$0$-statement}, which says that $G_{n,p}$ a.a.s.\ fails to satisfy the desired property if $p$ is below some threshold.

In recent years, there has been a great deal of work on transferring combinatorial theorems, such as Ramsey's theorem or Tur\'an's theorem~\cite{MR0018405}, to sparse random settings.  As a consequence, several new proofs of the $1$-statement of \cref{thm:RR} have been found.  Two such proofs were first given by Conlon--Gowers~\cite{MR3548529} and, independently, by Friedgut--R\"odl--Schacht~\cite{MR2760356} (see also Schacht~\cite{MR3548528}) with the use of their transference principles.  More recently, Nenadov and Steger~\cite{MR3438289} found a very short proof of the 1-statement of \cref{thm:RR} that uses the hypergraph container method of Saxton--Thomason~\cite{MR3385638} and Balogh--Morris--Samotij~\cite{MR3327533}.

However, these techniques are not suitable for proving the respective 0-statements such as that in \cref{thm:RR}.  Furthermore, whereas the 0-statement of the aforementioned sparse random analogue of Tur\'an's theorem is very easy to establish, proving the 0-statement of \cref{thm:RR} requires a significant amount of work.  To understand this, suppose that $G$ is some graph that is Ramsey for $H$ in $r$ colors.  As is well-known (see e.g.\ \cite[Theorem 3.4]{MR1782847}), the probability that $G_{n,p}$ contains $G$ as a subgraph is bounded away from zero if (and only if) $p = \Omega(n^{-1/m(G)})$, where $m(G)$ is the \emph{maximal density} of $G$, defined by
\[
	m(G) \coloneqq \max \left\{ \frac{e_J}{v_J}: J\subseteq G, v_J \geq 1\right\}.
\]
In particular, if $m(G) \leq m_2(H)$, then the $0$-statement of \cref{thm:RR} cannot hold.  Therefore, a prerequisite for any proof of the $0$-statement is the following result, which R\"odl--Ruci\'nski~\cite{MR1249720} termed the \emph{deterministic lemma}: If $G$ is Ramsey for $H$ in $r$ colors, then $m(G) > m_2(H)$.  We stress that this result is by no means trivial; in particular, it turns out to be false if we remove the assumption that $H$ is not a forest \cite{MR1276825,MR1768845}, or if we move from graphs to hypergraphs \cite{MR3725732}.

To complement the deterministic lemma, R\"odl--Ruci\'nski also proved what they termed a \emph{probabilistic lemma}. Loosely speaking, this is a result that says that the $0$-statement of \cref{thm:RR} is actually \emph{equivalent} to the deterministic lemma. In other words, an obvious necessary condition for the validity of the $0$-statement---the non-existence of a graph $G$ that is Ramsey for $H$ and satisfies $m(G) \leq m_2(H)$---is also a sufficient condition.

\subsection{Asymmetric Ramsey properties of random graphs}
Given our good understanding of Ramsey properties of random graphs in the symmetric case, provided by \cref{thm:RR}, it is natural to ask what happens if we remove the assumption that $H_1=\dotsb=H_r$. This question was first raised by Kohayakawa and Kreuter~\cite{MR1609513}, who proposed a natural conjecture for the threshold controlling when $G_{n,p}$ is Ramsey for an arbitrary tuple $(H_1,\dotsc,H_r)$. To state their conjecture, we need the notion of the \emph{mixed $2$-density}:  For graphs $H_1,H_2$ with $m_2(H_1) \geq m_2(H_2)$, their mixed $2$-density is defined as
\[
	m_2(H_1,H_2) \coloneqq \max \left\{ \frac{e_{J}}{v_{J}-2+1/m_2(H_2)} : J \subseteq H_1, v_J \geq 2 \right\}.
\]
With this terminology, we may state the conjecture of Kohayakawa and Kreuter~\cite{MR1609513}.
\begin{conjecture}[Kohayakawa--Kreuter~\cite{MR1609513}]\label{conj:KK}
  Let $H_1,\dots,H_r$ be graphs satisfying $m_2(H_1) \geq \dotsb \geq m_2(H_r)$ and $m_2(H_2) >1$. There exist constants $c,C>0$ such that
  \[
    \lim_{n \to \infty}\pr(G_{n,p}\text{ is Ramsey for } (H_1,\dotsc,H_r)) = 
    \begin{cases}
      1&\text{if }p \geq Cn^{-1/m_2(H_1,H_2)},\\
      0&\text{if }p \leq cn^{-1/m_2(H_1,H_2)}.
    \end{cases}
  \]
\end{conjecture}
The assumption $m_2(H_2) >1$ is equivalent to requiring that $H_1$ and $H_2$ are not forests; 
it was added by Kohayakawa, Schacht, and Sp\"ohel~\cite{MR3143588} to rule out sporadic counterexamples, in analogy with the assumption that $H$ is not a forest in \cref{thm:RR}.

The role of the mixed $2$-density $m_2(H_1, H_2)$ in the context of \cref{conj:KK} can seem a little mysterious at first, but there is a natural (heuristic) explanation.  Since one can color all edges that do not lie in a copy of $H_1$ with color $1$, the only important edges are those that do lie in copies of $H_1$.  The mixed $2$-density is defined in such a way that $p =\Theta(n^{-1/m_2(H_1,H_2)})$ is the threshold at which the number of copies of (the densest subgraph of) each of $H_2, \dotsc, H_r$ is at least of the same order of magnitude as the number of edges in the union of all copies of (the densest subgraph of) $H_1$ in $G_{n,p}$.  Since at least one edge in each copy of $H_1$ must receive a color from $\{2, \dotsc, r\}$, this is the point where avoiding monochromatic copies of $H_2, \dotsc, H_r$ becomes difficult.

\cref{conj:KK} has received a great deal of attention over the years, and has been proved in a number of special cases.  Following a sequence of partial results~\cite{MR1609513, MR2531778, MR3143588, MR3725732, MR3899160}, the $1$-statement of \cref{conj:KK} was proved by Mousset, Nenadov, and Samotij~\cite{MR4173138} with the use of the container method as well as a randomized ``typing'' procedure.  We henceforth focus on the $0$-statement, where progress has been more limited.

Note that, in order to prove the $0$-statement, one can make several simplifying assumptions. First, one can assume that $r$, the number of colors, is equal to $2$. Indeed, if one can a.a.s.\ $2$-color the edges of $G_{n,p}$ and avoid monochromatic copies of $H_1,H_2$ in colors $1,2$, respectively, then certainly $G_{n,p}$ is not Ramsey for $(H_1,\dots,H_r)$. Furthermore, if $H_2' \subseteq H_2$ is a subgraph satisfying $m_2(H_2')=m_2(H_2)$, then the $0$-statement for the pair $(H_1,H_2')$ implies the $0$-statement for $(H_1,H_2)$, as any coloring with no monochromatic copy of $H_2'$ in particular has no monochromatic copy of $H_2$. Thus, we may assume that $H_2$ is \emph{strictly $2$-balanced}, meaning that $m_2(H_2') < m_2(H_2)$ for any $H_2' \subsetneq H_2$. For exactly the same reason, we may assume that $H_1$ is \emph{strictly $m_2(\cdot,H_2)$-balanced}, meaning that $m_2(H_1',H_2)<m_2(H_1,H_2)$ for any $H_1' \subsetneq H_1$. Let us say that the pair $(H_1,H_2)$ is \emph{strictly balanced} if $H_2$ is strictly $2$-balanced and $H_1$ is strictly $m_2(\cdot,H_2)$-balanced. Additionally, let us say that $(H_1',H_2')$ is a \emph{strictly balanced pair of subgraphs} of $(H_1,H_2)$ if $(H_1',H_2')$ is strictly balanced and satisfies $m_2(H_2')=m_2(H_2)$ and $m_2(H_1',H_2')=m_2(H_1,H_2)$. All previous works on the $0$-statement of \cref{conj:KK} have made these simplifying assumptions, working in the case $r=2$ and with a strictly balanced pair $(H_1,H_2)$.

The original paper of Kohayakawa and Kreuter~\cite{MR1609513} proved the $0$-statement of \cref{conj:KK} when $H_1$ and $H_2$ are cycles.  This was extended to the case when both $H_1$ and $H_2$ are cliques in \cite{MR2531778}, and to the case when $H_1$ is a clique and $H_2$ is a cycle in \cite{MR4597167}. To date, the most general result is due to Hyde \cite{MR4565396}, who proved the $0$-statement of \cref{conj:KK} for almost all pairs of regular graphs $(H_1,H_2)$; in fact, this follows from Hyde's main result~\cite[Theorem~1.9]{MR4565396}, which establishes a certain deterministic condition whose validity implies the $0$-statement of \cref{conj:KK}.  Finally, the first two authors \cite{2305.19964} recently proved the $0$-statement of \cref{conj:KK} in the case where $m_2(H_1)=m_2(H_2)$. Because of this, we henceforth focus on the case that $m_2(H_1)>m_2(H_2)$.

\subsection{New results}

As in the symmetric setting, a necessary prerequisite for proving the $0$-statement of \cref{conj:KK} is proving the following \emph{deterministic lemma}: If $G$ is Ramsey for $(H_1,H_2)$, then $m(G) > m_2(H_1,H_2)$. The main result in this paper is a corresponding probabilistic lemma, which states that this obvious necessary condition is also sufficient.
\begin{theorem}\label{thm:prob-lemma}
	The $0$-statement of \cref{conj:KK} holds if and only if, for every strictly balanced pair $(H_1,H_2)$, every graph $G$ that is Ramsey for $(H_1, H_2)$ satisfies $m(G) > m_2(H_1, H_2)$.
\end{theorem}
More precisely, we prove that if $(H_1,H_2)$ is any pair of graphs and $(H_1',H_2')$ is a strictly balanced pair of subgraphs of $(H_1,H_2)$, then the $0$-statement of \cref{conj:KK} holds for $(H_1,H_2)$ if every graph $G$ which is Ramsey for $(H_1',H_2')$ satisfies $m(G) > m_2(H_1',H_2') = m_2(H_1,H_2)$.

While we believe that the probabilistic lemma, \cref{thm:prob-lemma}, is our main contribution, we are able to prove the deterministic lemma in a wide range of cases. This implies that the $0$-statement of \cref{conj:KK} is true for almost all pairs of graphs. The most general statement we can prove is slightly tricky to state because of the necessity of passing to a strictly balanced pair of subgraphs; however, here is a representative example of our results, which avoids this technicality and still implies \cref{conj:KK} for almost all pairs of graphs. We state the more general result in \cref{thm:main} below.
\begin{theorem}\label{thm:m2-2.5}
	\cref{conj:KK} holds for all sequences $H_1,\dots,H_r$ of graphs satisfying 
	$m_2(H_1) \geq \dotsb \geq m_2(H_r)$ and $m_2(H_2) > \frac {11}5$.
      \end{theorem}

As discussed above, \cref{thm:m2-2.5} follows easily from \cref{thm:prob-lemma} and a deterministic lemma for strictly balanced pairs $(H_1,H_2)$ satisfying $m_2(H_1) \geq m_2(H_2) > \frac {11}5$.  The deterministic lemma in this setting is actually very straightforward and follows from standard coloring techniques.

Using a number of other coloring techniques, we can prove the deterministic lemma (and thus \cref{conj:KK}) in several additional cases, which we discuss below. However, let us first propose a conjecture, which we believe to be of independent interest, and whose resolution would immediately imply \cref{conj:KK} in all cases.
\begin{conjecture}\label{conj:partitioning}
	For any graph $G$, there exists a forest $F \subseteq G$ such that
	\[
		m_2(G \setminus F) \leq m(G).
	\]
\end{conjecture}
Here, $G \setminus F$ denotes the graph obtained from $G$ by deleting the edges of $F$ (but not deleting any vertices).
To give some intuition for \cref{conj:partitioning}, we note that $m(G) \leq m_2(G) \leq m(G) +1$ for any graph $G$, and that $m_2(F)=1$ for any forest $F$ which is not a matching. Thus, it is natural to expect that by deleting the edges of a forest, we could decrease $m_2(G)$ by roughly $1$. \cref{conj:partitioning} says that this is roughly the case, in that the deletion of an appropriately-chosen forest can decrease $m_2(G)$ to lie below $m(G)$.

Moreover, we note that \cref{conj:partitioning} easily implies the deterministic lemma in all cases\footnote{Recall that the case of $m_2(H_1)=m_2(H_2)$ was settled in \cite{2305.19964}, so we may freely make this assumption.} with $m_2(H_1)>m_2(H_2)$, and thus implies \cref{conj:KK}. Indeed, it is straightforward to verify in this case that $m_2(H_1) > m_2(H_1,H_2)$ (see \cref{lem:sandwiching m2} below). Now, suppose that $G$ is some graph with $m(G) \leq m_2(H_1,H_2) < m_2(H_1)$. If \cref{conj:partitioning} is true, we may partition the edges of $G$ into a forest $F$ and a graph $K$ with $m_2(K) \leq m(G) < m_2(H_1)$. This latter condition implies, in particular, that $K$ contains no copy of $H_1$. Additionally, by the assumption $m_2(H_2) >1$ in \cref{conj:KK}, we know that $H_2$ contains a cycle and thus $F$ contains no copy of $H_2$. In other words, coloring the edges of $K$ with color $1$ and the edges of $F$ with color $2$ witnesses that $G$ is not Ramsey for $(H_1,\dots,H_r)$. 

Because of this, it would be of great interest to prove \cref{conj:partitioning}. Somewhat surprisingly, we know how to prove \cref{conj:partitioning} under the extra assumption that $m(G)$ is an integer. This extra condition seems fairly artificial, but we do not know how to remove it---our technique uses tools from matroid theory that seem to break down once $m(G)$ is no longer an integer. We present this proof in \cref{sec:matroid-appendix}, in the hope that it may serve as a first step to the full resolution of \cref{conj:partitioning}, and thus \cref{conj:KK}.

Although we are not able to resolve \cref{conj:partitioning}, we do have a number of other techniques for proving the deterministic lemma, and thus \cref{conj:KK}, under certain assumptions. First, we are able to resolve the case when the number of colors is at least three and $m_2(H_2)=m_2(H_3)$.
\begin{theorem}\label{thm:3 colors}
	Let $H_1,\dots,H_r$ be a sequence of graphs with $r \geq 3$ and suppose that $m_2(H_1) \geq m_2(H_2) = m_2(H_3) \geq \dotsb \geq m_2(H_r)$ and $m_2(H_2)>1$. Then \cref{conj:KK} holds for $H_1,\dots,H_r$.
\end{theorem}

We can also prove \cref{conj:KK} in a number of additional cases, expressed in terms of the properties of (a strictly balanced pair of subgraphs of) the pair $(H_1, H_2)$ of two densest graphs.

Recall that the \emph{degeneracy} of $H$ is the maximum over all $J \subseteq H$ of the minimum degree of $J$.

\begin{theorem}\label{thm:main}
  Suppose that $(H_1,H_2)$ is strictly balanced. Suppose additionally that one of the following conditions holds:
  \begin{enumerate}[label=(\alph*)]
	\item\label{mainitem:bipartite} $\chi(H_2) \geq 3$, or
	\item\label{mainitem:m1} $H_2$ is not the union of two forests, or
	\item\label{mainitem:high-chi} $\chi(H_1) > m_2(H_1,H_2)+1$, or
	\item\label{mainitem:degen} $H_1$ has degeneracy at least $\flo{2m_2(H_1,H_2)}$, or
	\item\label{mainitem:biclique} $H_1=K_{s,t}$ for some $s,t\geq 2$, or 
	\item\label{mainitem:matroids} $m_2(H_1)>\cei{m_2(H_1,H_2)}$.
  \end{enumerate}
  In any of these cases, \cref{conj:KK} holds for $(H_1,H_2)$.
\end{theorem}
\begin{remark}
	The only graphs $H_2$ which do not satisfy \ref{mainitem:bipartite} or \ref{mainitem:m1} are sparse bipartite graphs, such as even cycles. 
	On the other hand, 
	\ref{mainitem:high-chi} applies whenever $H_1$ is a clique\footnote{Note that $m_2(H_1,H_2) \leq m_2(H_1)$, hence \ref{mainitem:high-chi} holds if $\chi(H_1)>m_2(H_1)+1$, and cliques satisfy $m_2(K_k)=\frac{k+1}{2}$.} or, more generally, a graph obtained from a clique by deleting few edges.	
	Moreover, 
	\ref{mainitem:degen} applies to reasonably dense graphs, as well as all $d$-regular bipartite graphs with $d \geq 2$, and \ref{mainitem:biclique} handles all cases when $H_1$ is a biclique\footnote{In fact, our proof of \ref{mainitem:biclique} applies to a larger class of graphs, which we call \emph{$(s,t)$-graphs}; see \cref{sec:det proof} for details.}. 
	Thus, very roughly speaking, the strictly balanced cases that remain open in \cref{conj:KK} are those in which $H_2$ is bipartite and very sparse and $H_1$ is not ``too dense''.

	Case \ref{mainitem:matroids} is somewhat stranger and it is not obvious that there exist graphs to which it applies. However, one can check that, for example, it applies if $H_1 = K_{3,3,3,3}$ and $H_2=C_8$, and that none of the other cases of \cref{thm:main} (or any of the earlier results on \cref{conj:KK}) apply in this case. However, the main reason we include \ref{mainitem:matroids} is that it is implied by our partial progress on \cref{conj:partitioning}; since we believe that this conjecture is the correct approach to settling \cref{conj:KK} in its entirety, we wanted to highlight \ref{mainitem:matroids}.

	We remark that, unfortunately, the conditions in \cref{thm:main} do not exhaust all cases.  While it is quite likely that simple additional arguments could resolve further cases, \cref{conj:partitioning} remains the only (conjectural) approach we have found to resolve \cref{conj:KK} in all cases.  Moreover, our proof of the probabilistic lemma implies that, in order to prove \cref{conj:KK} for a pair $(H_1,H_2)$, it is enough to prove the deterministic lemma for graphs $G$ of order not exceeding an explicit constant $K=K(H_1,H_2)$. In particular, the validity of \cref{conj:KK} for any specific pair of graphs reduces to a finite computation.
\end{remark}

\subsection{Ramsey properties of graph families}
All of the results discussed in the previous subsection hold in greater generality, when we replace $H_1,\dots,H_r$ with $r$ finite families of graphs. In addition to being interesting in its own right, such a generalization also has important consequences  in the original setting of \cref{conj:KK}; indeed, our proof of the three-color result, \cref{thm:3 colors}, relies on our ability to work with graph families. Before we state our more general results, we need the following definitions.

\begin{definition}
  Let $\H_1,\dots,\H_r$ be finite families of graphs. We say that a graph $G$ is \emph{Ramsey} for $(\H_1,\dots,\H_r)$ if every $r$-coloring of $E(G)$ contains a monochromatic copy of some $H_i \in \H_i$ in some color $i \in \br{r}$.
\end{definition}

We now define the appropriate generalizations of the notions of maximum $2$-density and mixed $2$-density to families of graphs.  First, given a finite family of graphs $\H$, we let
\[
  m_2(\H) \coloneqq \min_{H \in \H} m_2(H).
\]
Second, given a graph $H$ and a (finite) family $\L$ of graphs, we let
\[
  m_2(H, \L) \coloneqq  \max \left\{ \frac{e_{J}}{v_{J}-2+1/m_2(\L)} : J \subseteq H, v_J \geq 2 \right\}.
\]
Third, given two finite families of graphs $\H$ and $\L$ with $m_2(\H) \geq m_2(\L)$, we define
\[
  m_2(\H,\L) \coloneqq \min_{H \in \H} m_2(H, \L).
\]
Finally, continuing the terminology above, let us say that the pair $(\H,\L)$ is \emph{strictly balanced} if every graph in $\L$ is strictly $2$-balanced and every graph in $\H$ is strictly $m_2(\cdot,\L)$-balanced.

The following conjecture is a natural generalization of \cref{conj:KK} to families of graphs.
\begin{conjecture}[Kohayakawa--Kreuter conjecture for families]\label{conj:families}
	Let $\H_1,\dots,\H_r$ be finite families of graphs with $m_2(\H_1) \geq \dotsb \geq m_2(\H_r)$ and suppose that $m_2(\H_2) >1$. There exist constants $c,C>0$ such that
	\[
		\lim_{n \to \infty}\pr(G_{n,p}\text{ is Ramsey for } (\H_1,\dots,\H_r)) = 
		\begin{cases}
			1&\text{if }p \geq Cn^{-1/m_2(\H_1,\H_2)},\\
			0&\text{if }p \leq cn^{-1/m_2(\H_1,\H_2)}.
		\end{cases}
	\]
\end{conjecture}

Note that, for any $H_1 \in \H_1,\dots,H_r \in \H_r$, the property of being Ramsey for $(H_1,\dots,H_r)$ implies the property of being Ramsey for $(\H_1,\dots,\H_r)$. Therefore, the $1$-statement of \cref{conj:families} follows from the $1$-statement of \cref{conj:KK}, which we know to be true by the result of Mousset, Nenadov, and Samotij \cite{MR4173138}. 

The $0$-statement of \cref{conj:families} remains open; the only progress to date is due to the first two authors \cite{2305.19964}, who proved \cref{conj:families} whenever $m_2(\H_1)=m_2(\H_2)$. We make further progress on this conjecture: as in the case of single graphs, we prove a probabilistic lemma that reduces the $0$-statement to a deterministic lemma, which is clearly a necessary condition.
\begin{theorem}[Probabilistic lemma for families] \label{thm:prob}
	The $0$-statement of \cref{conj:families} holds if and only if, for every strictly balanced pair $(\H_1,\H_2)$ of finite families of graphs, every graph $G$ that is Ramsey for $(\H_1,\H_2)$ satisfies $m(G) > m_2(\H_1,\H_2)$.
\end{theorem}
As in \cref{thm:m2-2.5,thm:main}, we can prove the deterministic lemma for families in a wide variety of cases, namely when every graph $H_1 \in \H_1$ or every graph $H_2 \in \H_2$ satisfies one of the conditions in \cref{thm:main}. In particular, we resolve \cref{conj:families} in many cases. However, we believe that the right way to resolve \cref{conj:families} in its entirety is the same as the right way to resolve the original Kohayakawa--Kreuter conjecture, \cref{conj:KK}. Namely, if \cref{conj:partitioning} is true, then \cref{conj:families} is true for all families of graphs.

\subsection{Organization}
Most of the rest of this paper is dedicated to proving \cref{thm:prob}, and thus also \cref{thm:prob-lemma}. Our technique is inspired by recent work of the first two authors \cite{2305.19964}, who proved \cref{conj:families} in the case $m_2(\H_1)=m_2(\H_2)$. Therefore, we assume henceforth that $m_2(\H_1)>m_2(\H_2)$. We will now change notation and denote $\H_1=\H$ and $\H_2=\L$. The names stand for \emph{heavy} and \emph{light}, respectively, and are meant to remind the reader that $m_2(\L) < m_2(\H)$.  We also assume henceforth that $(\H,\L)$ is a strictly balanced pair of families.

The rest of this paper is organized as follows. In \cref{sec:outline}, we present a high-level overview of our proof of \cref{thm:prob}. \cref{sec:prelim} contains a number of preliminaries for the proof, including the definitions and basic properties of \emph{cores}---a fundamental notion in our approach---as well as several simple numerical lemmas. The proof of \cref{thm:prob} is carried out in detail in \cref{sec:prob proof}. In \cref{sec:det proof}, we prove the deterministic lemma under various assumptions, which yields \cref{thm:main,thm:m2-2.5} as well as their generalizations to families. We conclude with two appendices: \cref{sec:3 color} proves \cref{thm:3 colors} by explaining what in our proof needs to be adapted to deal with the three-color setting; and \cref{sec:matroid-appendix} presents our partial progress on \cref{conj:partitioning}.

\subsubsection*{Additional note}
As this paper was being written, we learned that very similar results were obtained independently by Bowtell, Hancock, and Hyde \cite{BHH23}, who also resolve \cref{conj:KK} in the vast majority of cases. As with this paper, they first prove a probabilistic lemma, showing that resolving the Kohayakawa--Kreuter conjecture is equivalent to proving a deterministic coloring result. By using a wider array of coloring techniques, they are able to prove more cases of \cref{conj:KK} than we can. Additionally, they consider a natural generalization of the Kohayakawa--Kreuter to uniform hypergraphs (a topic that we chose not to pursue here) and establish its $0$-statement for almost all pairs of hypergraphs; see also \cite{MR3725732} for more on such hypergraph questions.
In contrast, their work does not cover families of graphs, a generalization that falls out naturally from our approach. 

\subsubsection*{Acknowledgments}
We would like to thank Anita Liebenau and Let\'{\i}cia Mattos for fruitful discussions on Ramsey properties of random graphs.
We are also indebted to Candida Bowtell, Robert Hancock, and Joseph Hyde for sharing an early draft of their paper \cite{BHH23} with us, and for their many invaluable comments.

\section{Proof outline}\label{sec:outline}
We now sketch, at a very high level, the proof of the probabilistic lemma. Let us fix a strictly balanced pair of families $(\H,\L)$. We wish to upper-bound the probability that $G_{n,p}$ is Ramsey for $(\H,\L)$, where $p \leq cn^{-1/m_2(\H,\L)}$ for an appropriately chosen constant $c=c(\H,\L)>0$. Our approach is modeled on the recent proof of the $0$-statement of \cref{thm:RR} due to the first two authors \cite{2305.19964}; however, there are substantial additional difficulties that arise in the asymmetric setting.

One can immediately make several simplifying assumptions.  First, if $G_{n,p}$ is Ramsey for $(\H,\L)$, then there exists some $G \subseteq G_{n,p}$ that is \emph{minimally} Ramsey for $(\H,\L)$, in the sense that any proper subgraph $G' \subsetneq G$ is not Ramsey for $(\H,\L)$. It is not hard to show (see \cref{lem:ramsey is core} below) that every minimally Ramsey graph has a number of interesting properties.  In particular, if $G$ is minimally Ramsey, then every edge of $G$ lies in at least one copy of some $H \in \H$, and at least one copy of some $L\in \L$.  Our arguments will exploit a well-known strengthening of this property, which we call \emph{supporting a core}; see \cref{def:core} for the precise definition.

We would ideally like to union-bound over all possible minimally Ramsey graphs $G$ in order to show that a.a.s.\ none of them appears in $G_{n,p}$.  Unfortunately, there are potentially too many minimally Ramsey graphs for this to be possible.  To overcome this, we construct a smaller family $\S$ of subgraphs of $K_n$ such that every Ramsey graph $G$ contains some element of $\S$ as a subgraph.  Since $\S$ is much smaller than the family of minimally Ramsey graphs, we can effectively union-bound over $\S$.
This basic idea also underlies the container method \cite{MR3385638,MR3327533} and the recent work of Harel, Mousset, and Samotij on the upper tail problem for subgraph counts~\cite{MR4484206}.  The details here, however, are slightly subtle; there are actually three different types of graphs in $\S$ and a different union-bound argument is needed to handle each type.

We construct our family $\S$ with the use of an exploration process on minimally Ramsey graphs, each of which supports a core.  This exploration process starts with a fixed edge of $K_n$ and gradually adds to it copies of graphs in $\H \cup \L$.  As long as the subgraph $G' \subseteq G$ of explored edges is not yet all of $G$, we add to $G'$ a copy of some graph in $\H \cup \L$ that intersects $G'$ but is not fully contained in it.  By choosing this copy in a principled manner (more on this momentarily), we can ensure that $\S$ satisfies certain conditions which enable this union-bound argument.

Since our goal is to show that the final graph $G'$ is rather dense (and thus unlikely to appear in $G_{n,p}$), we always prefer to add copies of graphs in $\H$, as these boost the density of $G'$.  If there are no available copies of $H \in \H$, we explore along some $L \in \L$.  As $L$ may be very sparse, this can hurt us; however, the ``core'' property guarantees that each copy of $L$ comes with at least one copy of some $H \in \H$ per new edge. An elementary (but fairly involved) computation shows that the losses and the gains pencil out, which is the key fact showing that $\S$ has the desired properties.

\section{Preliminaries}\label{sec:prelim}
\subsection{Ramsey graphs and cores}

Given a graph $G$, denote by $\F_\H[G],\F_\L[G]$ the set of all copies of members of $\H,\L$, respectively, in $G$. We think of $\F_\H[G],\F_\L[G]$ as hypergraphs on the ground set $E(G)$; in particular, we think of an element of $\F_\H[G],\F_\L[G]$ as a collection of edges of $G$ that form a copy of some $H \in \H,L \in \L$, respectively. To highlight the (important) difference between the members of $\H \cup \L$ and their copies (i.e.\ the elements of $\F_\H[G] \cup \F_\L[G]$), we will denote the former by $H$ and $L$ and the latter by $\cop{H}$ and $\cop{L}$.

Given a graph $G$ and $\F_\H \subseteq \F_\H[G], \F_\L \subseteq \F_\L[G]$, we say that the tuple $(G,\F_\H,\F_\L)$ is \emph{Ramsey} if, for every two-coloring of $E(G)$, there is an element of $\F_\H$ that is monochromatic red or an element of $\F_\L$ that is monochromatic blue. In particular, we see that $G$ is Ramsey for $(\H,\L)$ if and only if $(G,\F_\H[G],\F_\L[G])$ is Ramsey.  Having said that, allowing tuples $(G, \F_\H, \F_\L)$ where $\F_\H$ and $\F_\L$ are proper subsets of $\F_\H[G]$ and $\F_\L[G]$, respectively, enables us to deduce further useful properties.  These are encapsulated in the following definition.

\begin{definition}\label{def:core}
	An \emph{$(\H,\L)$-core} (or \emph{core} for short) is a tuple $(G,\F_\H,\F_\L)$, where $G$ is a graph and $\F_\H \subseteq \F_\H[G], \F_\L \subseteq \F_\L[G]$, with the following properties:
	\begin{itemize}
		
		\item The hypergraph $\F_\H \cup \F_\L$ is connected and spans $E(G)$.
		\item For every $\cop{H} \in \F_\H$ and every edge $e \in \cop{H}$, there exists an $\cop{L} \in \F_\L$ such that $\cop{H} \cap \cop{L} = \{e\}$.
		\item For every $\cop{L} \in \F_\L$ and every edge $e \in \cop{L}$, there exists an $\cop{H} \in \F_\H$ such that $\cop{H} \cap \cop{L} = \{e\}$.
	\end{itemize}
	We say that $G$ \emph{supports a core} if there exist $\F_\H \subseteq \F_\H[G],\F_\L \subseteq \F_\L[G]$ such that $(G,\F_\H,\F_\L)$ is a core.
\end{definition}

The reason we care about cores is that minimal Ramsey graphs support cores, as shown in the following lemma.  Essentially the same lemma appears in the work of R\"odl and Ruci\'nski~\cite{MR1249720}, where it is given as an exercise.  The same idea was already used in several earlier works, including \cite[Claim~6]{MR1609513} and \cite[Lemma~4.1]{MR4597167}.

\begin{lemma}\label{lem:ramsey is core}
  Suppose that a graph $G$ is Ramsey for $(\H,\L)$, but none of its proper subgraphs are Ramsey for $(\H, \L)$.  Then $G$ supports an $(\H,\L)$-core.
\end{lemma}
\begin{proof}
  As $G$ is Ramsey for $(\H,\L)$, we know that $(G,\F_\H[G],\F_\L[G])$ is a Ramsey tuple. Let $\F_\H \subseteq \F_\H[G], \F_\L \subseteq \F_\L[G]$ be inclusion-minimal subfamilies such that $(G,\F_\H,\F_\L)$ is still a Ramsey tuple.  In other words, this tuple is Ramsey, but for any $\F_\H' \subseteq \F_\H, \F_\L' \subseteq \F_\L$ such that at least one inclusion is strict, the tuple $(G,\F_\H',\F_\L')$ is not Ramsey.  We will show that $(G, \F_\H, \F_\L)$ is a core.

  If some $e \in E(G)$ is not contained in any edge of $\F_\H \cup \F_\L$, then $(G \setminus e, \F_\H, \F_\L)$ is still Ramsey, and thus $G \setminus e$ is Ramsey for $(\H,\L)$, contradicting the minimality of $G$.  Furthermore, if $\F_\H \cup \F_\L$ is not connected, then at least one of its connected components induces a Ramsey tuple, which contradicts the minimality of $(\F_\H, \F_\L)$.  Thus, the first condition in the definition of a core is satisfied. We now turn to the next two parts of the definition.

  To see that the second condition in the definition of a core is satisfied, fix some $\cop{H} \in \F_\H$ and some $e \in \cop{H}$.  By minimality, we can find a two-coloring of $E(G)$ such that no element of $\F_\L$ is blue and no element of $\F_\H \setminus \{\cop{H}\}$ is red.  Note that all edges of $\cop{H}$ are colored red, as otherwise our coloring would witness $(G, \F_\H, \F_\L)$ being not Ramsey.  Flip the color of $e$ from red to blue.  Since $\cop{H}$ is now no longer monochromatic red, we must have created a monochromatic blue element $\cop{L}$ of $\F_\L$.  As all edges of $\cop{H} \setminus e$ are still red, we see that $\cop{H} \cap \cop{L} = \{e\}$, as required. Interchanging the roles of $\F_\H,\F_\L$, and the colors yields the third condition in the definition of a core.
\end{proof}

\subsection{Numerical lemmas}
In this section, we collect a few useful numerical lemmas, all of which are simple combinatorial facts about vertex- and edge-counts in graphs. We begin with the following well-known result, which we will use throughout.
\begin{lemma}[The mediant inequality]\label{lem:mediant}
  Let $a,c \geq 0$ and $b,d>0$ be real numbers with $a/b \leq c/d$. Then
  \[
    \frac ab \leq \frac{a+c}{b+d} \leq \frac cd.
  \]
  Moreover, if one inequality is strict, then so is the other (which happens if and only if $a/b < c/d$).
\end{lemma}
\begin{proof}
  Both inequalities are easily seen to be equivalent to the inequality $ad \leq bc$, which is itself the same as $a/b \leq c/d$.
\end{proof}

\begin{lemma}\label{lem:sandwiching m2}
  Let $(\H,\L)$ be a strictly balanced pair. If $m_2(\L)< m_2(\H)$, then $m_2(\L) < m_2(\H,\L) < m_2(\H)$.
\end{lemma}
\begin{proof}
  To see the second inequality, let $H \in \H$ be a graph with $m_2(H) = m_2(\H)$ and observe that the strict $m_2(\cdot, \L)$-balancedness of $H$ implies that
  \[
    m_2(H,\L) = \frac{e_H}{v_H-2+1/m_2(\L)} = \frac{(e_H-1)+1}{(v_H-2) + 1/m_2(\L)} \leq \frac{m_2(H) \cdot (v_H-2) + 1}{(v_H-2)+1/m_2(\L)}.
  \]
  Since $m_2(H) = m_2(\H) > m_2(\L)$, \cref{lem:mediant} implies that $m_2(\H, \L) \leq m_2(H, \L) < m_2(\H)$.

  For the first inequality, let $H \in \H$ be a graph for which $m_2(H,\L) = m_2(\H,\L)$ and let $J \subseteq H$ be its subgraph with $\frac{e_J-1}{v_J-2} = m_2(H)$.  By the strict $m_2(\cdot, \L)$-balancedness of $H$, we have
  \[
    m_2(H, \L) \geq m_2(J, \L) = \frac{(e_J-1)+1}{(v_J-2)+1/m_2(\L)} = \frac{m_2(H) \cdot (v_J-2)+1}{(v_J-2)+1/m_2(\L)}.
  \]
  Since $m_2(H) > m_2(\L)$, \cref{lem:mediant} implies that $m_2(\H,\L) = m_2(H,\L) \geq m_2(J, \L) > m_2(\L)$.
\end{proof}

\begin{lemma}\label{lem:>alpha}
  Let $H \in \H$ be strictly $m_2(\cdot,\L)$-balanced. Then for any $F \subsetneq H$ with $v_F \geq 2$, we have
  \[
    e_H - e_F > m_2(H, \L) \cdot (v_H-v_F) \geq m_2(\H,\L) \cdot (v_H - v_F).
  \]
\end{lemma}

\begin{proof}
  The second inequality follows from the definition of $m_2(\H,\L)$.
  Since $e_F < e_H$, we may assume that $v_F < v_H$, as otherwise the
  claimed inequality holds vacuously.
  Since $H$ is strictly $m_2(\cdot,\L)$-balanced,
  we have
  \[
    m_2(H,\L) = \frac{e_H}{v_H - 2 + 1/m_2(\L)} = \frac{(e_H-e_F)+e_F}{(v_H-v_F) + (v_F -2 + 1/m_2(\L))}
  \]
  whereas
  \[
    \frac{e_F}{v_F - 2 + 1/m_2(\L)} < m_2(H,\L).
  \]
  Since $v_H > v_F$, we may use \cref{lem:mediant} to conclude that $(e_H -e_F)/(v_H - v_F) > m_2(H,\L)$.
\end{proof}

\begin{lemma}\label{lem:m2L}
  Let $L\in \L$ be strictly 2-balanced. Then for any $J \subsetneq L$ with $e_L \geq 1$, we have
  \[
    e_L-e_J \geq m_2(L) \cdot (v_L-v_J) \geq m_2(\L)\cdot (v_L-v_J).
  \]
  Moreover, the first inequality is strict unless $J = K_2$.
\end{lemma}
\begin{proof}
	The second inequality is immediate since $m_2(\L) \leq m_2(L)$.
  Since $e_J < e_L$, we may assume that $v_J < v_L$, as otherwise the
  claimed (strict) inequality holds vacuously.
	We clearly have equality if $J = K_2$ and strict inequality if $v_J=2$ and $e_J = 0$, so we may assume henceforth that $v_J >2$.
        Since $L$ is strictly $2$-balanced,
	\[
          m_2(L) = \frac{e_L-1}{v_L-2} = \frac{(e_L-e_J) + (e_J-1)}{(v_L-v_J)+(v_J-2)}
	\]
        whereas $(e_J-1)/(v_J-2) < m_2(L)$. Since $v_J>2$, we may apply \cref{lem:mediant} to conclude the desired result, with a strict inequality.
\end{proof}
\begin{lemma}\label{lem:X-vs-alpha}
  Suppose that $(\H,\L)$ is a strictly balanced pair.
  Defining $\alpha \coloneqq m_2(\H,\L)$ and $X \coloneqq \min_{H \in \H} \{(e_H-1)-\alpha \cdot (v_H-2)\}$, we have that
  \[
	X + (v_K-2)(\alpha-1) \geq e_K \cdot \left( \frac{\alpha}{m_2(L)}-1 \right)
  \]
  for every $L \in \L$ and every non-empty $K \subseteq L$. Moreover, the inequality is strict unless $K=K_2$.
\end{lemma}
\begin{proof}
  Without loss of generality, we may assume that $m_2(L) < \alpha$ and that $v_K > 2$, as otherwise the statement holds vacuously (recall from \cref{lem:sandwiching m2} that $\alpha = m_2(\H,\L) > m_2(\L) > 1$).
  Fix some $L \in \L$ and a nonempty $K \subseteq L$.
  Recall that each $H \in \H$ is strictly $m_2(\cdot,\L)$-balanced and satisfies $m_2(H,\L) \geq m_2(\H,\L)=\alpha$. This implies that
	\[
		\frac{e_H}{v_H-2+1/m_2(\L)} \geq \alpha
	\]
	or, equivalently,
	\[
          e_H \geq \alpha \cdot (v_H-2) + \frac{\alpha}{m_2(\L)}.
        \]
        Consequently,
	\[
		X = \min_{H \in \H} \{(e_H -1) - \alpha \cdot (v_H-2) \} \geq \frac{\alpha}{m_2(\L)}-1 \geq \frac{\alpha}{m_2(L)}-1,
	\]
	where the final inequality uses that $m_2(L)\geq m_2(\L)$.
	
        Since $L$ is strictly $2$-balanced and we assumed that $m_2(L) < \alpha$, we have
	\[
		(e_K-1) \cdot \left( \frac{\alpha}{m_2(L)}-1 \right)
		\leq m_2(L) \cdot (v_K-2) \cdot \left( \frac{\alpha}{m_2(L)}-1 \right)
		 = (v_K-2) (\alpha-m_2(L)).
	\]
	Rearranging the above inequality, we obtain
	\begin{align*}
		e_K \cdot \left( \frac{\alpha}{m_2(L)}-1 \right) - (v_K-2)(\alpha-1)
		&\leq (1-m_2(L))(v_K-2) +  \left( \frac{\alpha}{m_2(L)}-1 \right)\\
		&< \frac{\alpha}{m_2(L)}-1 \leq X,
	\end{align*}
        where the penultimate inequality uses the assumption that $v_K > 2$.
\end{proof}

\section{Proof of the probabilistic lemma}\label{sec:prob proof}
In this section, we prove \cref{thm:prob}. We in fact prove the following more precise statement. 
\begin{lemma}[\cref{thm:prob}, rephrased]\label{lem:prob rephrased}
  Let $(\H,\L)$ be a strictly balanced pair of finite families of graphs satisfying $m_2(\H)>m_2(\L)$. There exists a constant $c>0$ such that the following holds. If $p \leq cn^{-1/m_2(\H,\L)}$, then a.a.s.\ every $G \subseteq G_{n,p}$ which supports a core satisfies $m(G) \leq m_2(\H,\L)$.
\end{lemma}
Note that this immediately implies the difficult direction in \cref{thm:prob}. Indeed, suppose that the $0$-statement of \ref{conj:families} fails for some tuple $(\H_1, \dotsc, \H_r)$, i.e., the random graph $G_{n,p}$ is Ramsey for $(\H_1, \dotsc, \H_r)$ with probability bounded away from zero when $p = cn^{-1/m_2(\H_1,\H_2)}$, for an arbitrarily small constant $c>0$.  In particular, with probability bounded away from zero, $G_{n,p}$ contains a graph that is also Ramsey for any pair $(\H,\L)$ of families of subgraphs of $(\H_1,\H_2)$.  For an appropriately chosen pair $(\H,\L)$, \cref{lem:ramsey is core} implies that some subgraph $G \subseteq G_{n,p}$ supports an $(\H,\L)$-core.  By the assumed assertion of \cref{lem:prob rephrased}, a.a.s.\ any such $G \subseteq G_{n,p}$ satisfies $m(G) \leq m_2(\H,\L)$.  However, by the deterministic lemma (i.e.\ the assumption of \cref{thm:prob}), we know that no such $G$ can be Ramsey for $(\H,\L)$, a contradiction.

Our proof of \cref{lem:prob rephrased} follows closely the proof of the probabilistic lemma in recent work of the first two authors \cite{2305.19964}. Fix a strictly balanced pair $(\H,\L)$ of families satisfying $m_2(\H)>m_2(\L)$, and let $\alpha \coloneqq m_2(\H,\L)$. 
Let $\Gbad$ denote the set of graphs $G \subseteq K_n$ which support a core and satisfy $m(G) > m_2(\H,\L)$. The key lemma, which implies \cref{lem:prob rephrased}, is as follows. 

\begin{lemma}\label{lem:exists-S}
  There exist constants $\Lambda,K>0$ and a collection $\S$ of subgraphs of $K_n$ satisfying the following properties:
  \begin{enumerate}[label=(\alph*)]
  \item
    \label{item:prob-a} Every element of $\Gbad$ contains some $S \in \S$ as a subgraph.
  \item
    \label{item:prob-balance} Every $S \in \S$ satisfies at least one of the following three conditions:
    \begin{enumerate}[label=(\roman*)]
    \item \label{item:large} $v_S \geq \log n$ and $e_S \geq \alpha \cdot (v_S-2)$;
    \item\label{item:many degen} $v_S < \log n$ and $e_S \geq \alpha \cdot v_S + 1$;
    \item\label{item:found core} $v_S \leq K$ and $m(S) > \alpha$.
    \end{enumerate} 
    
  \item
    \label{item:prob-count} For every $k \in \br n$, there are at most $(\Lambda n)^k$ graphs  $S \in \S$ with $v_S = k$.
  \end{enumerate}
\end{lemma}

Before we prove \cref{lem:exists-S}, let us see why it implies \cref{lem:prob rephrased}.
\begin{proof}[Proof of \cref{lem:prob rephrased}]
	Recall that $p \leq cn^{-1/\alpha}$, for a small constant $c=c(\H,\L)$ to be chosen later.
	We wish to prove that a.a.s.\ $G_{n,p}$ contains no element of $\Gbad$. By \cref{lem:exists-S}\ref{item:prob-a}, it suffices to prove that a.a.s.\ $G_{n,p}$ contains no element of $\S$. By \ref{item:prob-balance}, the elements of $\S$ are of three types, each of which we deal with separately. First, recall that for any fixed graph $S$ with $m(S) > \alpha$, we have that $\pr(S \subseteq G_{n,p})=o(1)$ (see e.g.\ \cite[Theorem 3.4]{MR1782847}). As there are only a constant number of graphs on at most $K$ vertices, we may apply the union bound and conclude that a.a.s.\ no graph $S$ satisfying $v_S \leq K$ and $m(S) > \alpha$ appears in $G_{n,p}$. This deals with the elements of $\S$ corresponding to case \ref{item:found core}.

	Let $\S' \subseteq \S$ be the set of $S \in \S$ which lie in cases \ref{item:large} or \ref{item:many degen}. We have that
	\begin{align*}
		\pr(S \subseteq G_{n,p} \text{ for some }S \in \S') &\leq \sum_{S \in \S'} p^{e_S}\\
		&\leq \sum_{k=1}^{\cei{\log n}-1} (\Lambda n)^k p^{\alpha k + 1} + \sum_{k=\cei{\log n}}^\infty (\Lambda n)^k p^{\alpha(k-2)}\\
		&\leq p\sum_{k=1}^{\infty} (\Lambda c^\alpha)^k  + c^{-2\alpha}n^2 \sum_{k=\cei{\log n}}^\infty (\Lambda c^\alpha)^k
	\end{align*}
	We now choose $c$ so that $\Lambda c^\alpha =e^{-3}$. Then the first sum above can be bounded by $p$, which tends to $0$ as $n \to \infty$. The second term can be bounded by $2c^{-2\alpha} n^{-1}$, which also tends to $0$ as $n \to \infty$. All in all, we find that a.a.s.\ $G_{n,p}$ does not contain any graph in $\S$, as claimed.
\end{proof}

\subsection{The exploration process and the proof of Lemma \ref{lem:exists-S}}
In this section, we prove \cref{lem:exists-S}.
We will construct the family $\S$ by considering an exploration process on the set $\G$ of graphs $G \subseteq K_n$ which support a core.  For each such $G \in \G$, let us arbitrarily choose collections $\F_\H \subseteq \F_\H[G]$ and $\F_\L \subseteq \F_\L[G]$ such that $(G,\F_\H,\F_\L)$ is a core.  From now on, by copies of graphs from $\H,\L$ in $G$, we mean only those copies that belong to the families $\F_\H, \F_\L$, respectively.  This subtlety will be extremely important in parts of the analysis.

We first fix arbitrary orderings on the graphs in $\H$ and $\L$.
Additionally, we fix a labeling of the vertices of $K_n$, which induces an ordering of all subgraphs according to the lexicographic order. Together with the ordering on $\H,\L$, we obtain a lexicographic ordering on all copies in $K_n$ of graphs in $\H,\L$.   Now, given a $G \in \G$, we build a sequence $G_0 \subsetneq G_1 \subsetneq \dotsb \subseteq G$ as follows.  We start with $G_0$ being the graph comprising only the smallest edge of $G$.  As long as $G_i \neq G$, do the following: Since $G \neq G_i$ and $(G, \F_\H, \F_\L)$ is a core, there must be some copy of a graph from $\H \cup \L$ which belongs to $\F_\H \cup \F_\L$ that intersects $G_i$ but is not fully contained in $G_i$.  Call such an \emph{overlapping} copy \emph{regular} if it intersects $G_i$ in exactly one edge, called its \emph{root}; otherwise, call the copy \emph{degenerate}.  
We form $G_{i+1}$ from $G_i$ as follows:
\begin{enumerate}
\item
  Suppose first that there is an overlapping copy of some graph in $\H$.
  We form $G_{i+1}$ by adding to $G_i$ the smallest (according to the lexicographic order) such copy.
  We call $G_i \to G_{i+1}$ a \emph{degenerate $\H$-step}.
\item
  Otherwise, there must be an overlapping copy $\cop{L}$ of some $L\in \L$.  Note that, for every edge $e \in \cop{L} \setminus G_i$, there must be a copy of some $H \in \H$ that meets $\cop{L}$ only at $e$, as $(G, \F_\H, \F_\L)$ is a core.  Note further that this copy of $H$ does not intersect $G_i$, as otherwise we would perform a degenerate $\H$-step.  We pick the smallest such copy for every $e \in \cop{L} \setminus G_i$, and call it $\cop{H_e}$ (note that the graphs $H_e \in \H$ such that $H_e \cong \cop{H_e}$ may be different for different choices of $e$).  We say that $\cop{L}$ is \emph{pristine} if it is regular and the graphs $\{\cop{H_e}\}_{e \in \cop{L} \setminus G_i}$ are all vertex-disjoint (apart from the intersections that they are forced to have in $V(\cop{L})$).
  \begin{enumerate}[label=(\theenumi.\arabic*)]
  \item
    If there is a pristine copy of some graph in $\L$, we pick the smallest one in the following sense:
    First, among all edges of $G_i$ that are roots of a pristine copy of some graph in $\L$, we choose the one that arrived to $G_i$ earliest.  Second, among all pristine copies that are rooted at this edge, we pick the smallest (according to the lexicographic order).
    We then form $G_{i+1}$ by adding to $G_i$ this smallest copy $\cop{L}$ as well as all $\cop{H_e}$ where $e \in \cop{L} \setminus G_i$.
    We call $G_i \to G_{i+1}$ a \emph{pristine step}.
  \item
    If there are no pristine copies of any graph in $\L$, we pick the smallest (according to the lexicographic order) overlapping copy $\cop{L}$ of a graph in $\L$ and we still form $G_{i+1}$ by adding to $G_i$ the union of $\cop{L}$ and all its $\cop{H_e}$ with $e \in \cop{L} \setminus G_i$.
    We call $G_i \to G_{i+1}$ a \emph{degenerate $\L$-step}.
  \end{enumerate}
\end{enumerate}
We define the \emph{balance} of $G_i$ to be
\[
  b(G_i) \coloneqq e_{G_i} - \alpha \cdot v_{G_i},
\]
where we recall that $\alpha = m_2(\H,\L)$. The key result we will need in order to prove~\ref{item:prob-balance} is the following lemma. We remark that a similar result was proved by Hyde~\cite[Claims~6.2 and~6.3]{MR4565396}; it plays an integral role in his approach to the Kohayakawa--Kreuter conjecture.

\begin{lemma}\label{lem:increase-balance}
  For every $i$, we have that $b(G_{i+1}) \geq b(G_i)$. Moreover,
  there exists some $\delta=\delta(\H,\L)>0$ such that $b(G_{i+1}) \geq b(G_i)+\delta$ if $G_{i+1}$ was obtained from $G_i$ by a degenerate step.
\end{lemma}
As the proof of \cref{lem:increase-balance} is somewhat technical, we defer it to \cref{sec:increase}.
For the moment, we assume the result and continue the discussion of how we construct the family $\S$.
We now let $\Gamma \coloneqq \cei{2\alpha/\delta}$, where $\delta$ is the constant from \cref{lem:increase-balance}.
For $G \in \G$, let
\[
  \tau(G) \coloneqq \min\{i :  v_{G_i} \geq \log n \text{ or } G_i = G \text{ or }G_{i-1} \to G_i\text{ is the $\Gamma$th degenerate step}\}
\]
and let
\begin{equation}\label{eq:S-def}
  \S \coloneqq \{G_{\tau(G)} : G \in \Gbad\}.
\end{equation}

Having defined the family $\S$, we are ready to prove \cref{lem:exists-S}. Since the definition of $\S$ clearly guarantees property \ref{item:prob-a}, it remains to establish properties \ref{item:prob-balance} and \ref{item:prob-count}.
We begin by showing that, if $K$ is sufficiently large (depending only on $\H$ and $\L$), then \ref{item:prob-balance} holds.
\begin{proof}[Proof of \cref{lem:exists-S}\ref{item:prob-balance}]
  Let $\delta$ be the constant from \cref{lem:increase-balance}, let $M \coloneqq \max \{e_L \cdot v_H: H \in \H, L \in \L\}$, and let $K \coloneqq 2M^2 \Gamma$; note that each of these parameters depends only on $\H$ and~$\L$.

  Every $S \in \S$ is of the form $G_{\tau(G)}$ for some $G \in \Gbad$.
  We split into cases depending on which of the three conditions defining $\tau(G)$ caused us to stop the exploration.
  Suppose first that we stopped the exploration because $v_S \geq \log n$.
  By \cref{lem:increase-balance}, we have that
  \[
    e_S - \alpha \cdot v_S= b(S) = b(G_{\tau(G)}) \geq b(G_0) = 1-2\alpha,
  \]
  and therefore $e_S \geq \alpha \cdot (v_S-2)$. This yields case \ref{item:large}.

  Next, suppose we stopped the exploration because step $G_{\tau(G)-1} \to G_{\tau(G)}$ was the $\Gamma$th degenerate step.  As we are not in the previous case, we may assume that $v_S <\log n$. By \cref{lem:increase-balance} and our choice of $\Gamma$, we have that
  \[
    e_S - \alpha \cdot v_S =b(S) = b(G_{\tau(G)}) \geq b(G_0) + \Gamma \delta \geq 1-2\alpha + 2\alpha = 1.
  \]
  Rearranging, we see that $e_S \geq \alpha \cdot v_S + 1$, yielding case \ref{item:many degen}.
  
  The remaining case is when we stop because $S = G \in \Gbad$.  Since the definition of $\Gbad$ implies that $m(G) > \alpha$, in order to establish \ref{item:found core}, we only need to show that $v_G \leq K$.
  For this proof, we need to keep track of another parameter during the exploration process, which we term the \emph{pristine boundary}.  Recall that at every pristine step, we add to $G_i$ a copy $\cop{L}$ of some $L \in \L$ that intersects $G_i$ in a single edge (the root), and then add copies $\cop{H_e}$ of graphs $H_e \in \H$, one for every edge of $\cop{L}$ apart from the root.  Let us say that the \emph{boundary} of this step is the set of all newly added vertices that are not in $\cop{L}$, that is, the set $V(G_{i+1}) \setminus (V(G_i) \cup V(\cop{L})) = (\bigcup_{e \in \cop{L} \setminus G_i} V(\cop{H_e})) \setminus V(\cop{L})$.  Note that the size of the boundary is equal to
  \[
    Y_i \coloneqq \sum_{e \in \cop{L} \setminus G_i} (v_{H_e}-2);
  \]
  indeed, by the definition of pristine steps, the copies $\cop{H_e}$ are vertex-disjoint outside of $V(\cop L)$.
  
  We claim that $Y_i \geq 3$.  To see this, note first that $L$ has at least three edges, as it is not a forest. Similarly, each $H_e$ has at least three vertices.  Putting these together, we see that there are at least two terms in the sum, and every term in the sum is at least one.  Thus, $Y_i \geq 3$ unless $e_L=3$ and $v_{H_e}=3$ for all $e$.  But in this case, $L = K_3 = H_e \in \H$ for all $e$, which means that $\cop{L}$ should have been added to $G_i$ as a degenerate $\H$-step.

  We now inductively define the pristine boundary $\partial G_i$ of $G_i$ as follows.  We set $\partial G_0\coloneqq \varnothing$. 
  If $G_i \to G_{i+1}$ is a pristine step, then we delete from $\partial G_i$ the two endpoints of the root and add to $\partial G_i$ the boundary of this pristine step.  Note that $\ab{\partial G_{i+1}} \geq \ab{\partial G_i} + Y_i - 2 \geq \ab{\partial G_i} + 1$. On the other hand,
  if $G_i \to G_{i+1}$ is a degenerate step, then we only remove vertices from $\partial G_i$, without adding any new vertices. Namely, we remove from $\partial G_i$ all the vertices which are included in the newly added graphs. In other words, if we performed a degenerate $\H$-step by adding a copy $\cop{H}$ of some graph in $\H$, we set $\partial G_{i+1} \coloneqq \partial G_i \setminus V(\cop{H})$.  Similarly, if we performed a degenerate $\L$-step by adding a copy $\cop{L}$ of some graph in $\L$ along with the graphs $\cop{H_e}$ for all $e \in \cop{L} \setminus G_i$, we set $\partial G_{i+1} \coloneqq \partial G_i \setminus (V(\cop{L}) \cup \bigcup_e V(\cop{H_e}))$.  Note that in either case $\ab{\partial G_{i+1}} \geq \ab{\partial G_i} - M$, as the union of all graphs added in each degenerate step can have at most $M$ vertices.

  We now argue that $\partial G_{\tau(G)}=\varnothing$.  Indeed,
  suppose we had some vertex $v \in \partial G_{\tau(G)}$.  By
  definition, $v$ was added during a pristine step, as a vertex of a
  copy $\cop{H_e}$ of some graph $H_e \in \H$, and was never touched again.
  Observe that $v$ is incident to some edge $uv$ of $\cop{H_e}$ that was not touched by any later step of the exploration.
  However, as $(G, \F_\H, \F_\L)$ is a core and $\cop{H_e} \in \F_\H$, there must be some $\cop{L_{uv}} \in \F_\L$ that intersects $\cop{H_e}$ only at $uv$.  Moreover, as $\cop{L_{uv}}$ has minimum degree at least two (by the strict $2$-balancedness assumption), there is some edge $vw \in \cop{L_{uv}} \setminus uv$ that is incident to $v$.  Since we assumed that $G_{\tau(G)} = G$, the edge $vw$ must have been added at some point, a contradiction to the assumption that $v$ was never touched again.

  Finally, since $\ab{\partial G_i}$ increases by at least one during every pristine step and decreases by at most $M$ during each of the at most $\Gamma$ degenerate steps,  in order to achieve $\partial G_{\tau(G)}=\varnothing$, there can be at most $M\Gamma$ pristine steps.  In particular, the total number of exploration steps is at most $M\Gamma + \Gamma$.  As each exploration step adds at most $M$ vertices to $G_i$, we conclude that $v_G \leq M(M\Gamma + \Gamma)+2 \leq K$.  This completes the proof of \ref{item:found core}.
\end{proof}

\begin{proof}[Proof of \cref{lem:exists-S}\ref{item:prob-count}]
  Suppose $S$ has $k$ vertices and let $G \in \Gbad$ be such that $S = G_{\tau(G)}$.  We consider the exploration process on $G$.  Note that in every step we add an overlapping copy of a graph from a finite family $\F$ that comprises all graphs in $\H$ (for the cases where we made a degenerate $\H$-step) and graphs in $\L$ that have graphs from $\H$ glued on subsets of their edges, with all intersection patterns (for the pristine and degenerate $\L$-steps). Let $\F^\times$ denote the graphs in $\F$ that correspond to a pristine step.

  Now, every degenerate step can be described by specifying the graph $F \in \F$ whose copy $\cop{F}$ we are adding, the subgraph $F' \subseteq F$ and the embedding $\varphi \colon V(F') \to V(G_i)$ that describe the intersection $\cop{F} \cap G_i$, and the sequence of $v_F - v_{F'}$ vertices of $K_n$ that complete $\varphi$ to an embedding of $F$ into $K_n$.  Every pristine step is uniquely described by the root edge in $G_i$, the graph $F \in \F^\times$, the edge of $F$ corresponding to the root, and the (ordered sequence of) $v_F - 2$ vertices of $K_n$ that complete the root edge to a copy of $F$ in $K_n$.
There are at most $n^k$ ways to choose the sequence of vertices that were added through this exploration process, in the order that they are introduced to $G$. Each pristine step adds at least one new vertex, so there are at most $k$ pristine steps. Furthermore, there are always at most $\Gamma$ degenerate steps, meaning that $\tau(G) \leq k + \Gamma$. In particular, there are at most $(k + \Gamma) \cdot 2^{k + \Gamma}$ ways to choose $\tau(G)$ and to specify which steps were pristine.

For every degenerate step, there are at most \[\sum_{F \in \F} \sum_{\ell=2}^{v_F} \binom{v_F}{\ell} k^\ell \leq \ab{\F} \cdot (k+1)^{M_v}\] ways of choosing $F \in \F$ and describing the intersection of its copy $\cop{F}$ with $G_i$ (the set $V(F') \subseteq V(F)$ and the embedding $\varphi$ above), where $M_v \coloneqq \max \{v_F \colon F \in \F \}$. As for the pristine steps, note that, in the course of our exploration, the sequence of the arrival times of the roots to $G_{\tau(G)}$ must be non-decreasing.  This is because as soon as an edge appears in some $G_i$, every pristine step that includes it as a root at any later step is already available, and we always choose the one rooted at the edge that arrived to $G$ the earliest. Therefore, there are at most $\binom{e_S + k}{k}$ possible sequences of root edges, since this is the number of non-decreasing sequences of length $k$ in $\br{e_S}$. To supplement this bound, remember that every step increases the number of edges in $G_i$ by at most $M_e \coloneqq \max \{e_F : F \in \F \}$, which means that \[e_S \leq 1 + \tau(G)\cdot M_e \leq 1 + (k + \Gamma) \cdot M_e.\] To summarize, the number of $S \in \S$ with $k$ vertices is at most \[n^k \cdot (k+ \Gamma) \cdot 2^{k + \Gamma} \cdot \left( \ab{\F} \cdot (k+1)^{M_v}\right)^\Gamma \cdot \binom{(k+\Gamma) \cdot M_e + k + 1}{k} \cdot \left(\ab{\F} \cdot M_e \right)^k.\] 
Every term in this product, apart from the first, is bounded by an exponential function of $k$, since $\Gamma, \ab \F, M_v$, and $M_e$ are all constants. Therefore, if we choose $\Lambda = \Lambda(\H,\L)$ sufficiently large, we find that the number of $S \in \S$ with $v_S=k$ is at most $(\Lambda n)^k$, as claimed.
\end{proof}

\subsection{Proof of Lemma \ref{lem:increase-balance}}\label{sec:increase}
In this section, we prove \cref{lem:increase-balance}. The proof is
divided into a number of claims. 
Recall \cref{lem:>alpha}, which asserts that 
\[
	e_H-e_F > m_2(\H,\L) \cdot (v_H-v_F) = \alpha \cdot (v_H - v_F)
\]
for all $H \in \H$ and all $F \subsetneq H$. This implies that we can choose some $\delta_1=\delta_1(\H,\L)>0$ so that
\begin{equation}\label{eq:delta1}
	e_H-e_F \geq \alpha \cdot(v_H-v_F) + \delta_1
\end{equation}
for all $H \in \H$ and all $F \subsetneq H$; we henceforth fix such a $\delta_1>0$.

Our first claim deals with the (easy)
case that $G_i \to G_{i+1}$ is a degenerate $\H$-step.
\begin{claim}\label{claim:delta1}
  If $G_i \to G_{i+1}$ is a degenerate $\H$-step, then $b(G_{i+1}) \geq
  b(G_i) + \delta_1$.
\end{claim}
\begin{proof}
  Suppose we add to $G_i$ a copy of some $H \in \H$ that intersects
  $G_i$ on a subgraph $F \subseteq H$. This means that 
  \[
    e_{G_{i+1}} = e_{G_i} + (e_H - e_F) \qquad \text{ and } \qquad v_{G_{i+1}} = v_{G_i} + (v_H - v_F)
  \]
  and thus
  \[
    b(G_{i+1}) - b(G_i) = (e_H - e_F) - \alpha\cdot (v_H - v_F) \geq \delta_1,
  \]
  where the inequality follows from \eqref{eq:delta1}, as
  $F$ must be a proper subgraph of $H$.
\end{proof}

Now, suppose that $G_i \to G_{i+1}$ is an $\L$-step, either degenerate
or pristine, which means that we add a copy $\cop{L}$ of some $L
\in \L$ and then add, for every edge $e \in \cop{L} \setminus G_i$, a
copy $\cop{H_e}$ of some $H_e \in \H$.
Let $G_i' \coloneqq G_i \cup \cop{L}$ and let $\cop{J} \coloneqq G_i
\cap \cop{L}$, so that $\cop{J} \cong J$ for some $J \subsetneq L$
with at least one edge. Note that
\begin{equation}
  \label{eq:bal-Gi'-Gi}
  b(G_i') - b(G_i) = (e_L - e_J) - \alpha \cdot (v_L - v_J),
\end{equation}
as we add $e_L -e_J$ edges and $v_L - v_J$ vertices to $G_i$ when forming $G_i'$.

In order to analyze $b(G_{i+1}) - b(G_i')$, we now define an auxiliary graph $\I$ as follows.  Its vertices are the
edges of $\cop{L} \setminus \cop{J}$.  Recall that, for every such
edge $e$, the graph $\cop{H_e} \cong H_e$ intersects $G_i'$ only in
the edge $e$.  A pair $e, f$ of edges of $\cop{L} \setminus \cop{J}$
will be adjacent in $\I$ if and only if their corresponding graphs
$\cop{H_e}$ and $\cop{H_f}$ share at least one edge (equivalently, the graphs $\cop{H_e} \setminus e$ and $\cop{H_f} \setminus f$ share an edge).

Denote the connected components of $\I$ by $K_1, \dotsc, K_m$ and note
that each of them corresponds to a subgraph of $\cop{L} \setminus
\cop{J}$.  For each $j \in \br{m}$, let
\[
  U_j \coloneqq \bigcup_{e \in K_j} (\cop{H_e} \setminus e).
\]
Note that the graphs $G_i'$ and $U_1, \dotsc, U_m$ are pairwise
edge-disjoint and that each $U_j$ shares at least
$v_{K_j}$ vertices (the endpoints of all the edges of $K_j$) with
$G_i'$.  It follows that
\begin{equation}
  \label{eq:H-steps-after-L}
  b(G_{i+1}) - b(G_i') \geq \sum_{j=1}^m (e_{U_j} - \alpha \cdot
  (v_{U_j} - v_{K_j})) = \sum_{j=1}^m (b(U_j) + \alpha \cdot v_{K_j}).
\end{equation}
Finally, as in the statement of \cref{lem:X-vs-alpha}, define
\[
  X \coloneqq \min \{(e_H-1)-\alpha \cdot (v_H-2) : H \in \H\}.
\]
The following claim lies at the heart of the matter.

\begin{claim}\label{claim:component}
  For every $j \in \br{m}$, we have
  \[
    b(U_j) \geq X - 2\alpha - (v_{K_j}-2) + \min\{\delta_1,1\} \cdot \mathbf{1}_{v_{K_j}>2}.
  \]
\end{claim}
\begin{proof}
  Since $K_j$ is connected in $\I$, we may order its edges as $e_1,
  \dotsc, e_\ell$ so that, for each $r \in \br{\ell-1}$, the edge
  $e_{r+1}$ is $\I$-adjacent to $\{e_1, \dotsc, e_r\}$.
  Letting $F \subseteq H_{e_{r+1}}$ be the subgraph corresponding to this intersection,
  we define, for each $r \in \{0, \dotsc, \ell\}$,
  \[
    U_j^r \coloneqq \bigcup_{s=1}^r (\cop{H_{e_s}} \setminus e_s),
  \]
  so that $\varnothing = U_j^0 \subseteq \dotsb \subseteq U_j^\ell =
  U_j$.  Observe that
  \[
    b(U_j^1) = e_{U_j^1} - \alpha \cdot v_{U_j^1} = (e_{H_{e_1}} -1) -
    \alpha \cdot v_{H_{e_1}} \geq X - 2\alpha,
  \]
  where the inequality follows from the definition of $X$.

  Suppose now that $r \geq 1$ and let $\cop{F}$ be the intersection of
  $\cop{H_{e_{r+1}}} \setminus e_{r+1}$ with $U_j^r$; note that this intersection is non-empty as $e_{r+1}$ is $\I$-adjacent to $\{e_1,\dots,e_r\}$. We have
  \[
    b(U_j^{r+1}) - b(U_j^r) = (e_{H_{e_{r+1}}}-1-e_F) - \alpha
    \cdot (v_{H_{e_{r+1}}} - v_F).
  \]
  Let $t_{r+1}$ be the number of endpoints of $e_{r+1}$ that are not
  in $U_j^r$.
  Suppose first that $t_{r+1} = 0$, that is, both endpoints of $e_{r+1}$ are already in $U_j^r$.
  In this case, both endpoints of $e_{r+1}$ also belong to $\cop{F}$
  and thus $\cop{F} \cup e_{r+1}$ is isomorphic to a
  subgraph $F^+ \subseteq H_{e_{r+1}}$ with $e_F+1$ edges and $v_F$ vertices, which means that
  \[
    b(U_j^{r+1}) - b(U_j^r) = (e_{H_{e_{r+1}}} - e_{F^+}) - \alpha
    \cdot (v_{H_{e_{r+1}}} - v_{F^+}) \geq 0,
  \]
  by \cref{lem:>alpha}.
  In case $t_{r+1} > 0$, $F$ is a proper subgraph of $H_{e_{r+1}}$ and thus we have
  \[
    b(U_j^{r+1}) - b(U_j^r) \geq \delta_1-1 \geq \delta_1-t_{r+1},
  \]
  see~\eqref{eq:delta1}. We may thus conclude that
  \[
    b(U_j) = b(U_j^1) + \sum_{r=1}^{\ell-1} (b(U_j^{r+1}) - b(U_j^r))
    \geq X - 2\alpha - \sum_{r=1}^{\ell-1} t_{r+1} + \delta_1 \cdot
    \mathbf{1}_{t_2+\dotsb+t_\ell>0}.
  \]
  The desired inequality follows as $t_2 + \dotsb + t_{\ell} = \ab{V(K_j) \setminus V(U_j^1)} \leq
  v_{K_j}-2$ and, further, $v_{K_j} > 2$ implies that the sum
  $t_2+\dotsb+t_r$ is either positive or at most $v_{K_j}-3$.
\end{proof}
We are now ready to show that the balance only increases when we perform an $\L$-step.
\begin{claim}\label{claim:delta3}
  If $G_i \to G_{i+1}$ is an $\L$-step, then
  $b(G_{i+1}) \geq b(G_i)$. Moreover, if this $\L$-step is degenerate,
  then $b(G_{i+1}) \geq b(G_i) + \delta_2$ for some $\delta_2>0$ that
  depends only on $\H$ and $\L$.
\end{claim}
\begin{proof}
  By~\eqref{eq:bal-Gi'-Gi}, \eqref{eq:H-steps-after-L}, and \cref{claim:delta3}, we have
  \begin{align*}
    b&(G_{i+1}) - b(G_i) = b(G_i')-b(G_i) + b(G_{i+1})-b(G_i') \\
               &\geq (e_L-e_J) - \alpha \cdot (v_L-v_J) + \sum_{j=1}^m
                 (b(U_j) + \alpha \cdot v_{K_j}) \\
               &\geq
                 (e_L-e_J) - \alpha \cdot (v_L-v_J) + \sum_{j=1}^m \left( X + (v_{K_j}-2)(\alpha-1) \right) + \min\{\delta_1,1\} \cdot \mathbf{1}_{\I \neq \varnothing},
  \end{align*}
  since $\I$ is nonempty only if one of its components has more than two vertices.
	We now apply \cref{lem:X-vs-alpha} to each component $K_j$ to conclude that
	\begin{equation*}
          \sum_{j=1}^m
          \left(X + (v_{K_j}-2)(\alpha-1)\right) \geq \sum_{j=1}^m e_{K_j} \cdot  \left( \frac{\alpha}{m_2(L)}-1 \right) = (e_L-e_J ) \left( \frac{\alpha}{m_2(L)}-1 \right).
	\end{equation*}
	Therefore,
	\[
		b(G_{i+1}) -b(G_i) \geq  (e_L-e_J) \cdot \frac{\alpha}{m_2(L)}-\alpha \cdot (v_L-v_J) + \min\{\delta_1, 1\} \cdot \mathbf{1}_{\I \neq \varnothing} \geq \min\{\delta_1, 1\} \cdot \mathbf{1}_{\I \neq \varnothing},
	\]
	where the last inequality follows from \cref{lem:m2L}.
	This implies the desired result if the $\L$-step is pristine.
        If the $\L$-step is not pristine but $\I$ has no edges, it means that some vertex was repeated between different $\cop{H_e}$.
	In that case, the first inequality
        in~\eqref{eq:H-steps-after-L} is strict (we assumed there that
        the graphs $U_j$ share no vertices outside of $V(K_j)$).
        All in all, we obtain the desired boost in the degenerate case.
\end{proof}

Combining \cref{claim:delta1,claim:delta3}, we obtain \cref{lem:increase-balance}. This completes the proof of the probabilistic lemma.

\section{Proof of the deterministic lemma}\label{sec:det proof}

Given the probabilistic lemma and the work of the first two authors on the symmetric case~\cite{2305.19964}, in order to prove \cref{conj:families}, which generalizes the Kohayakawa--Kreuter conjecture, we only need to show the following. For every strictly balanced pair $(\H,\L)$ of finite families of graphs with $m_2(\H) > m_2(\L) > 1$, we can two-color the edges of every graph $G$ satisfying $m(G) \leq m_2(\H,\L)$ so that there are neither red monochromatic copies of any $H \in \H$ nor blue monochromatic copies of any $L \in \L$.
As discussed in the introduction, we do not know how to do this in all cases. However, the following proposition lists a number of extra assumptions under which we are able to find such a coloring. We recall the notion of the \emph{$1$-density} (or \emph{fractional arboricity}) of a graph $L$, defined by
\[
  m_1(L) \coloneqq \max \left\{ \frac{e_J}{v_J-1} : J \subseteq L, v_J \geq 2 \right\}.
\]
We also make the following definition.
\begin{definition} Given positive integers $s \leq t$, we say that a graph is an \emph{$(s,t)$-graph} if its minimum degree is at least $s$, and
	every edge contains a vertex of degree at least $t$.
We say that a graph is \emph{$(s,t)$-avoiding} if none of its subgraphs is an $(s,t)$-graph.
\end{definition}
\begin{proposition}\label{prop:deterministic}
  Let $(\H,\L)$ be a strictly balanced pair of finite families of graphs satisfying $m_2(\H)>m_2(\L)$ and suppose that at least one of the following conditions holds:
	\begin{enumerate}[label=(\alph*)]
	\item\label{item:bipartite} $\chi(L) \geq 3$ for all $L \in \L$;
	\item\label{item:high chromatic} $\chi(H) > m_2(\H,\L) + 1$  for every $H \in \H$;
	\item\label{item:forests} $m_1(L) > 2$ for all $L \in \L$;
	\item \label{item:biclique} 
          every $H \in \H$ contains an $(s,t)$-graph as a subgraph, for some integers $s \leq t$ satisfying
	\[
		\frac{1}{s+1} + \frac{1}{t+1} < \frac{1}{m_2(\H,\L)};
	\]
	\item\label{item:strong partitioning} $\cei{m_2(\H,\L)}<m_2(\H)$;
	\end{enumerate}
	Then any graph $G$ with $m(G) \leq m_2(\H,\L)$ is not Ramsey for $(\H,\L)$.
\end{proposition}
Cases \ref{item:bipartite}--\ref{item:forests} all follow fairly easily from known coloring techniques;  we supply the details in the remainder of this section.
Case \ref{item:biclique} is proved by a short inductive argument, see below.
Case \ref{item:strong partitioning} follows from our partial progress on \cref{conj:partitioning}, namely, that we are able to prove it when $m(G)$ is an integer; we present the proof of this result in \cref{sec:matroid-appendix}.  We end this section with short derivations of \cref{thm:main,thm:m2-2.5} from the proposition.

\begin{proof}[Proof of~\cref{thm:m2-2.5}]
  Assume that $m_2(L)> \frac{11}5$.  By passing to a subgraph with the same $2$-density, we may assume that $L$ is strictly $2$-balanced.  Thanks to cases \ref{item:bipartite} and \ref{item:forests} of \cref{prop:deterministic}, we are done unless $m_1(L) \leq 2$ and $L$ is bipartite.  The bounds on $m_1(L)$ and $m_2(L)$ imply that $2v_L-2 \geq e_L > \frac{11}5(v_L-2)+1$, which yields $v_L < 7$.  However, as $L$ is bipartite on at most six vertices, we have $m_2(L)\leq m_2(K_{3,3})=2$, a contradiction.  
\end{proof}

\begin{proof}[Proof of \cref{thm:main}]
  Cases \ref{mainitem:bipartite}, \ref{mainitem:m1}, \ref{mainitem:high-chi}, and \ref{mainitem:matroids} follow immediately\footnote{\cref{prop:deterministic}\ref{item:forests} implies \cref{thm:main}\ref{mainitem:m1} thanks to Nash-Williams's theorem (\cref{thm:nash-williams} below).} from \cref{prop:deterministic}. For \cref{thm:main}\ref{mainitem:degen}, note that a graph with minimum degree $d$ is a $(d,d)$-graph. Thus, if $H_1$ has degeneracy at least $d$, then it contains some $(d,d)$-graph as a subgraph. Similarly, \cref{thm:main}\ref{mainitem:biclique} follows, since if $s \leq t$, then $K_{s,t}$ is an $(s,t)$-graph satisfying $1/m_2(K_{s,t})=(s+t-2)/(st-1) \geq 1/(s+1)+1/(t+1)$.
\end{proof}

\subsection{Auxiliary results}
\label{sec:auxiliary-results}

We start with a helpful observation relating $m(G)$ and the degeneracy of $G$. We say that a graph is \emph{$d$-degenerate} if its degeneracy is at most $d$.

\begin{lemma}
  \label{lemma:degeneracy}
  Every graph $G$ is $\flo{2m(G)}$-degenerate.
\end{lemma}
\begin{proof}
  For every $G' \subseteq G$, we have \[\delta(G') \leq \left\lfloor\frac{2e_{G'}}{v_{G'}}\right\rfloor \leq \flo{2m(G)},\]
where $\delta(G')$ is the minimum degree of $G'$. 
\end{proof}

Our second lemma allows us to compare between the various densities.

\begin{lemma}\label{lem:density ineqs}
  For every graph $H$, we have $m_2(H) \leq m_1(H) + \frac{1}{2} \leq m(H) + 1$.
\end{lemma}
\begin{proof}
Notice that both $\frac{e-1}{v-2} \leq \frac{e}{v-1} + \frac{1}{2}$ and $\frac{e}{v-1} \leq \frac{e}{v} + \frac{1}{2}$ are equivalent to $e \leq \binom v 2$, so both inequalities hold whenever $v,e$ are the numbers of vertices and edges, respectively, of any graph.  In particular, if $v,e$ correspond to the subgraph of $H$ that achieves $m_2(H)$, we find that $m_2(H)=\frac{e-1}{v-2} \leq \frac{e}{v-1} + \frac{1}{2} \leq m_1(H) + \frac{1}{2}$. The second inequality follows in the same way, now passing to the subgraph that achieves $m_1(H)$.
\end{proof}

Our next lemma gives a lower bound on the average degree of an $(s,t)$-graph.  We remark that this inequality is tight for $K_{s,t}$ and that it can be restated as $e_H/v_H \geq m(K_{s,t})$.

\begin{lemma}
  \label{lem:st-graph}
  If $H$ is an $(s,t)$-graph, then
  \[\frac{1}{s} + \frac{1}{t} \geq \frac{v_H}{e_H}.\]
\end{lemma}

\begin{proof}
  The assumption that $H$ is an $(s,t)$-graph implies that, for every $uv \in E(H)$, we have $1/\deg(u) + 1/\deg(v) \leq 1/s + 1/t$.  This means that
  \[
    e_H \cdot \left(\frac{1}{s} + \frac{1}{t}\right) \geq \sum_{uv \in H} \left( \frac{1}{\deg(u)} + \frac{1}{\deg(v)}\right) = v_H.\qedhere
  \]
\end{proof}
The next lemma supplies a decomposition of a graph of bounded degeneracy.
\begin{lemma}
  \label{lemma:degenerate-partition}
  If a graph $G$ is $(dk-1)$-degenerate, for some positive integers $d,k$, then there is a partition $V(G) = V_1 \cup \dotsb \cup V_k$ such that the graphs $G[V_1], \dotsc, G[V_k]$ are all $(d-1)$-degenerate.
\end{lemma}
\begin{proof}
  We may construct the desired partition in the following way.  Initialize $V_1 = \dotsb = V_k = \varnothing$  and let $v_1, \dotsc, v_n$ be an ordering of the vertices of $G$ such that every $v_i$ has at most $dk-1$ neighbors preceding it.  We distribute the vertices one-by-one, each time putting $v_i$ in a set $V_j$ where, at the time, $v_i$ has the smallest number of neighbors.  By the pigeonhole principle, this number is at most $\flo{\frac{dk-1}{k}} = d-1$.
\end{proof}

Finally, we quote Nash-Williams's theorem on partitions of graphs into forests.

\begin{theorem}[Nash-Williams \cite{MR0161333}]
  \label{thm:nash-williams}
  A graph $G$ can be partitioned into $t$ forests if and only if $\lceil m_1(G) \rceil \leq t$.
\end{theorem}

\subsection{Proof of Proposition \ref{prop:deterministic}}
We are now ready to prove \cref{prop:deterministic}.  Denote $\alpha \coloneqq m_2(\H,\L)$ and let $G$ be an arbitrary graph satisfying $m(G) \leq \alpha$.  We will argue that (the edge set of) $G$ can be partitioned into an $\H$-free graph and an $\L$-free graph.  We split into cases, depending on which condition is satisfied by the pair $(\H,\L)$.

\subsubsection*{Cases~\ref{item:bipartite} and~\ref{item:high chromatic}.}
Let $k \coloneqq \flo{\alpha}+1$, so that $m(G) \leq \alpha < k$, and note that \cref{lemma:degeneracy} implies that $G$ is $(2k-1)$-degenerate.  Consequently, \cref{lemma:degenerate-partition} yields two partitions of the edges of $G$:  a partition into a $1$-degenerate graph and a $k$-colorable graph; and a partition into a $(k-1)$-degenerate graph and a bipartite graph.
The existence of the first partition proves~\ref{item:high chromatic}, as every $1$-degenerate graph is $\L$-free whereas the assumption on $\H$ implies that $\chi(H) > k$ for every $H \in \H$.  We now argue that the existence of the second partition proves~\ref{item:bipartite}.  To this end, note that the assumption there implies that every bipartite graph is $\L$-free, so it is enough to show that $\delta(H) \geq k$ for every $H \in \H$ and thus every $(k-1)$-degenerate graph is $\H$-free.  To see that this is the case, consider an arbitrary $H \in \H$ and let $v \in V(H)$ be its vertex with smallest degree.  As $H$ is strictly $m_2(\cdot, \L)$-balanced, \cref{lem:>alpha} gives $\delta(H) = e_H - e_{H\setminus v} > \alpha$, unless $v_H = 3$, in which case $H = K_3$ and we still have $\delta(H) \geq 2 = m_2(H) \geq m_2(\H) > \alpha$.  Since $\delta(H)$ is an integer, we actually have $\delta(H) \geq \flo{\alpha}+1 = k$, as needed.

\subsubsection*{Case~\ref{item:forests}.}
It is enough to show that $G$ can be partitioned into an $\H$-free graph and a union of two forests; indeed, if $m_1(L) > 2$ for all $L \in \L$, then no union of two forests can contain a member of $\L$ as a subgraph, by (the easy direction of) \cref{thm:nash-williams}.  Let $m_1(\H) \coloneqq \min\{m_1(H):H \in \H\}$. By \cref{lem:density ineqs} and the assumption $m(G) \leq m_2(\H,\L) < m_2(\H)$, we find that \[m_1(G) \leq m(G) + \frac{1}{2} \leq m_2(\H) + \frac{1}{2} \leq m_1(\H) + 1.\] As a result, if we let $t \coloneqq \cei{m_1(\H)}$, we find that $\cei{m_1(G)} \leq t+1$ and therefore \cref{thm:nash-williams} supplies a partition $G$ into $t+1$ forests $G_1, \dots ,G_{t+1}$.  Taking $G' \coloneqq G_1 \cup \dots \cup G_{t-1}$, we arrive at a partition $G = G' \cup (G_t \cup G_{t+1})$. By (the easy direction of) \cref{thm:nash-williams}, we know that $m_1(G') \leq t-1 < m_1(\H)$, so $G'$ is $\H$-free.  As $G_t$ and $G_{t+1}$ are forests, we get the desired decomposition.

\subsubsection*{Case~\ref{item:biclique}.}
It is enough to show that $G$ can be decomposed into a forest and an $(s,t)$-avoiding graph.
Assume that this is not the case and let $G$ be a smallest counterexample with $m(G) \leq \alpha$.  It is enough to show that $G$ is an $(s+1,t+1)$-graph, as then \cref{lem:st-graph} gives
\[
  \frac{1}{s+1} + \frac{1}{t+1} \geq \frac{v_G}{e_G} \geq \frac{1}{m(G)} \geq \frac{1}{\alpha},
\]
a contradiction.

Suppose first that $G$ has a vertex $v$ of degree at most $s$.  By minimality of $G$, we can decompose the edges of $G \setminus v$ into an $(s,t)$-avoiding graph $K$ and a forest $F$.  Adding an arbitrary edge incident with $v$ to $F$ and the remaining edges to $K$ maintains $F$ being a forest and $K$ being $(s,t)$-avoiding, as any $(s,t)$-subgraph of $K$ would have to use $v$, which has degree at most $s-1$ in $K$.  This contradicts our assumption on indecomposability of $G$.

Second, suppose that $G$ contains an edge $uv$ with $\deg(u), \deg(v) \leq t$.  By minimality of $G$, we can decompose $G' \coloneqq G\setminus uv$ into a forest $F$ and an $(s,t)$-avoiding graph $K$.  Adding $uv$ to $F$ must close a cycle, meaning that both $u$ and $v$ are incident to at least one $F$-edge of $G'$ and thus the $K$-degrees of $u$ and $v$ in $G'$ are at most $t-2$.  This means, however, that we can add $uv$ to $K$ while still keeping the degrees of both its endpoints strictly below $t$.  Again, we find that $K$ contains no $(s,t)$-subgraph, a contradiction.

\subsubsection*{Case~\ref{item:strong partitioning}.}
Let $k \coloneqq \cei{m_2(\H,\L)}$.  Since we assume that $m_2(\H) > k$, it is enough to decompose $G$ into a forest and a graph $K$ with $m_2(K) \leq k$.  The following theorem, which implies \cref{conj:partitioning} in the case that $m(G)$ is an integer, supplies such a decomposition.

\begin{theorem}\label{thm:integer partition}
  Let $k$ be an integer, and let $G$ be a graph with $m(G) \leq k$. Then there exists a forest $F \subseteq G$ such that $m_2(G \setminus F) \leq k$.
\end{theorem}

The proof of \cref{thm:integer partition} is substantially more involved, as it relies on techniques from matroid theory.  We are hopeful that similar techniques may be used to prove \cref{conj:partitioning} in its entirety. We defer the proof of \cref{thm:integer partition} to \cref{sec:matroid-appendix}.

\appendix

\section{The three-color setting}\label{sec:3 color}
In this section, we explain what about the proof needs to change to handle the case $r \geq 3$, and prove \cref{thm:3 colors}. As many of these results are essentially identical to the results discussed previously, we omit or shorten several of the proofs. We begin by defining a natural three-color analogue of cores.
\begin{definition} \label{def:3-color core}
	Let $\H_1,\H_2,\H_3$ be three finite families of graphs. A tuple $(G,\F_1,\F_2,\F_3)$ is an \emph{$(\H_1,\H_2,\H_3)$-core} if $G$ is a graph and $\F_i \subseteq \F_{\H_i}[G]$ for all $i \in \br 3$ are families satisfying the following properties:
	\begin{itemize}
		\item The hypergraph $\F_{1} \cup \F_{2} \cup \F_{3}$ is connected and spans $E(G)$.
		\item For every $i \in \br 3$, every $\cop{H_i} \in \F_i$, every edge $e \in \cop{H_i}$, and every $j \in \br 3 \setminus \{i\}$, there is some $\cop{H_j} \in \F_j$ with $\cop{H_i} \cap \cop{H_j} = \{e\}$.
	\end{itemize}
	We say that $G$ \emph{supports a core} if there exists a core $(G,\F_1,\F_2,\F_3)$.
\end{definition}
The following simple lemma is a straightforward generalization of \cref{lem:ramsey is core}, so we omit the proof.
\begin{lemma}\label{lem:3-color ramsey is core}
	Let $G$ be a graph that is minimally Ramsey for $(\H_1,\H_2,\H_3)$, in the sense that any proper subgraph $G' \subsetneq G$ is not Ramsey for $(\H_1,\H_2,\H_3)$. Then $G$ supports a core.
\end{lemma}

It would be very convenient if every $(\H_1,\H_2,\H_3)$-core  were also an $(\H_1,\H_2)$-core. At first glance this seems true, since the intersection property in \cref{def:3-color core} easily implies the intersection property in \cref{def:core}. Unfortunately, it may be the case that the hypergraph $\F_1 \cup \F_2 \cup \F_3$ is connected, but that the hypergraph $\F_1 \cup \F_2$ is disconnected. Nonetheless, this is the only obstruction, and the following result is true.
\begin{lemma}\label{lem:3-color core is 2-color core}
	Let $(G,\F_1,\F_2,\F_3)$ be $(\H_1,\H_2,\H_3)$-core for some families of graphs $\H_1,\H_2,\H_3$. Then $(G,\F_1,\F_2 \cup \F_3)$ is an $(\H_1,\H_2 \cup \H_3)$-core.
\end{lemma}
\begin{proof}
	First note that the hypergraph $\F_1 \cup (\F_2 \cup \F_3)$ is simply the same as the hypergraph $\F_1 \cup \F_2 \cup \F_3$, so it is connected and spans $E(G)$ by assumption. For every $\cop{H_1} \in \F_1$ and every edge $e \in \cop{H_1}$, we may apply \cref{def:3-color core} with $j=2$ to see that there exists some $\cop{H_2} \in \F_2 \subseteq \F_2 \cup \F_3$ such that $\cop{H_1} \cap \cop{H_2} = \{e\}$. Similarly, applying \cref{def:3-color core} with $j=1$, we see that for every $\cop{H_{23}} \in \F_2 \cup \F_3$ and every edge $e \in \cop{H_{23}}$, there is some $\cop{H_1} \in \F_1$ such that $\cop{H_1} \cap \cop{H_{23}} = \{e\}$. Thus, $(G,\F_1,\F_2 \cup \F_3)$ is an $(\H_1,\H_2 \cup \H_3)$-core.
\end{proof}
The key (trivial) observation is that if $m_2(\H_2) = m_2(\H_3)$, then $m_2(\H_2 \cup \H_3)$ is also equal to both these numbers, as $m_2(\H_2 \cup \H_3) = \min \{m_2(\H_2),m_2(\H_3)\}$. Now, suppose we are given families $\H_1,\H_2,\H_3$ with $m_2(\H_1)>m_2(\H_2)=m_2(\H_3)$. By passing to families of subgraphs, we may assume that $\H_2,\H_3$ are strictly 2-balanced and that $\H_1$ is strictly $m_2(\cdot,\H_2)$-balanced. We now define $\H=\H_1$ and $\L=\H_2 \cup \H_3$. By \cref{lem:prob rephrased}, we know that there exists some $c>0$ such that if $p \leq cn^{-1/m_2(\H,\L)}$, then a.a.s.\ $G_{n,p}$ contains no subgraph $G$ which supports an $(\H,\L)$-core and satisfies $m(G) >m_2(\H,\L)$.

On the other hand, if $G_{n,p}$ is Ramsey for $(\H_1,\H_2,\H_3)$, then it must contain some minimally Ramsey subgraph $G$. By \cref{lem:3-color core is 2-color core,lem:3-color ramsey is core}, $G$ supports an $(\H,\L)$-core. Moreover, by the above, we must have $m(G) \leq m_2(\H,\L) = m_2(\H_1,\H_2)$, for otherwise $G \nsubseteq G_{n,p}$ a.a.s. Given this, the following deterministic lemma concludes the proof.
\begin{lemma}
	Let $\H_1,\H_2,\H_3$ satisfy $m(\H_1) \geq m(\H_2) \geq m(\H_3)>1$. If $G$ is Ramsey for $(\H_1,\H_2,\H_3)$, then $m(G) > m_2(\H_1,\H_2)$.
\end{lemma}
\begin{proof}
	We will actually prove that $m(G) > m_2(\H_1)$, which implies the desired result since $m_2(\H_1) \geq m_2(\H_1,\H_2)$. Suppose for contradiction that $m(G) \leq m_2(\H_1)$. By \cref{thm:nash-williams} (cf.\ the proof of \cref{prop:deterministic}\ref{item:forests}), we know that $G$ is the union of an $\H_1$-free graph and two forests. As $m_2(\H_2) \geq m_2(\H_3)>1$, every graph in $\H_2 \cup \H_3$ contains a cycle, and hence each of these forests is $\H_2 \cup \H_3$-free. Using one color for the $\H_1$-free graph and one color for each of the two forests, we see that $G$ is not Ramsey for $(\H_1,\H_2,\H_3)$.
\end{proof}

\section{Proof of Conjecture \ref{conj:partitioning} in the integer case}
\label{sec:matroid-appendix}
In this section, we present the proof of \cref{thm:integer partition}, which implies \cref{conj:partitioning} in the case that $m(G)$ is an integer. We will use some well-known results from matroid theory; all definitions and proofs can be found in any standard reference on matroid theory, such as Oxley's book \cite{MR2849819}.

The main result we will need is the following matroid partitioning theorem, originally due to Edmonds \cite{MR0190025}. We remark that this theorem easily implies Nash-Williams's theorem (\cref{thm:nash-williams}), which was used in the proof of \cref{prop:deterministic}\ref{item:forests}.
\begin{theorem}\label{thm:matroid partition}
	Let $M_1,M_2$ be matroids on the same ground set $E$, with rank functions $r_1,r_2$, respectively. Then $E$ can be partitioned as $E=I_1 \cup I_2$, with $I_i$ independent in $M_i$ for $i=1,2$, if and only if
	\[
		r_1(X) + r_2(X) \geq \ab X
	\]
	for every $X \subseteq E$.
\end{theorem}
A slightly weaker statement appears as \cite[Theorem 11.3.12]{MR2849819}, where the result is only stated when $M_1=M_2$. However, it is clear and well-known that the same proof proves \cref{thm:matroid partition}, using the formula for the rank of a matroid union, as given in \cite[Theorem 11.3.1]{MR2849819}.

In our application, we will set $E=E(G)$ and let $M_1$ be the graphic matroid of $G$, whose independent sets are precisely the acyclic subgraphs of $G$. We may view any subset of $E(G)$ as a subgraph $J$ of $G$; we then use $e_J$ rather than $\ab J$ to denote the size of this subset of $E(G)$. Additionally, we use $v_J$ to denote the number of vertices incident to any edge of $J$, and $\omega_J$ to denote the number of connected components of $J$. It is well-known (e.g.\ \cite[equation 1.3.8]{MR2849819}) that the rank function of $M_1$ is given by $r_1(J) = v_J - \omega_J$
for all $J \subseteq E(G)$. 

The second matroid we use will be one whose independent sets are precisely those subgraphs $K \subseteq G$ with $m_2(K) \leq k$. The fact that this is a matroid is the content of the next lemma.
\begin{lemma}\label{lem:m2 is matroid}
	Let $G$ be a graph and let $k$ be a positive integer. Then the family of subgraphs $K \subseteq G$ with $m_2(K) \leq k$ is the collection of independent sets of a matroid.
\end{lemma}
\begin{proof}
	Define a function $f\colon 2^{E(G)} \to \Z$ by $f(J) = k(v_J-2)+1$, for every $J \subseteq E(G)$. Note that this function is integer-valued since $k \in \Z$. Additionally, it is clear that $f$ is increasing, in the sense that $f(J) \leq f(J')$ whenever $J \subseteq J'$. Finally, we claim that $f$ is submodular. This is easiest to see by recalling that the function $g(J) =v_J$ is submodular (see e.g.\ \cite[Proposition 11.1.6]{MR2849819}); as $f$ is obtained from $g$ by multiplying by a positive constant and adding a constant, we find that $f$ is submodular as well.

	Now, by \cite[Corollary 11.1.2]{MR2849819}, we find that there exists a matroid $M(f)$ on $E(G)$ whose independent sets are precisely those $K \subseteq E(G)$ with the property that $e_{J} \leq f({J})$ for all non-empty $J \subseteq K$. Note that, for a graph $J$ with at least three vertices, the inequality $e_{J} \leq f(J)$ is equivalent to $d_2(J) \leq k$, where $d_2(J)=(e_{J}-1)/(v_{J}-2)$. If $J$ is non-empty and has only two vertices, then it must have one edge and $e_{J} \leq f(J)$ always holds. Thus, we see that $K$ is independent in $M(f)$ if and only if $\max \{(e_{J}-1)/(v_{J}-2) :J \subseteq K, v_{J}\geq 3\} \leq k$. This condition is precisely the condition that $m_2(K) \leq k$.
\end{proof}
In order to apply \cref{thm:matroid partition} to the matroids $M_1,M_2$, we need a way of lower-bounding the rank function of $M_2$. This is achieved by the following lemma.
\begin{lemma}\label{lem:check-sc}
	Let $k$ be a positive integer. If $J$ is a graph with $m(J) \leq k$, then there is a subgraph $J' \subseteq J$ with $m_2(J') \leq k$ and $e_J \leq e_{J'} + v_J- 1.$
\end{lemma}
\begin{proof}
	A well-known theorem of Hakimi \cite{MR0180501}, which is itself a simple consequence of \cref{thm:matroid partition}, implies that 
	since $m(J) \leq k$,  we can partition $J$ into graphs $J_1,\dots,J_k$, with $m(J_i) \leq 1$ for all $i$ (i.e.\ every component of every $J_i$ has at most one cycle). We may assume without loss of generality that $J_k$ is non-empty. Let $e$ be an edge of $J_k$ and define $J' = J_1 \cup \dotsb \cup J_{k-1} \cup \{ e\}$. We claim that $m_2(J') \leq k$ and $e_{J} \leq e_{J'} + v_J -1$.

	The second claim is fairly easy to see, as
	\[
		e_{J'} =1+ \sum_{i=1}^{k-1} e_{J_i} = 1+(e_J - e_{J_k}) \geq 1+e_J - v_{J_k} \geq e_J - v_J +1,
	\]
	where the second equality uses the fact that $J_1,\dots,J_k$ partition $J$, and the two inequalities follow from $e_{J_k} \leq v_{J_k} \leq v_J$, since $m(J_k) \leq 1$ and $J_k \subseteq J$.

	So it remains to prove that $m_2(J') \leq k$, i.e.\ that $d_2(L) \leq k$ for all $L \subseteq J'$. If $v_{L} \leq 2k-1$, then
	\[
		d_2(L) \leq \frac{\binom{v_{L}}2-1}{v_{L}-2}=\frac 12 \cdot\frac{v_L^2-v_L-2}{v_L-2} = \frac 12 (v_L+1) \leq k,
	\]
	as claimed. So we may assume that $v_L \geq 2k$.
	As $m(J_i) \leq 1$ for all $i$, we see that $e_{L} \leq (k-1)v_{L}+1$. Therefore,
	\[
		d_2(L) = \frac{e_{L}-1}{v_{L}-2} \leq \frac{(k-1)v_{L}}{v_{L}-2} \leq \frac{kv_L -2k}{v_L-2} = k.\qedhere
	\]
\end{proof}
With all of these preliminaries, we are ready to prove \cref{thm:integer partition}.
\begin{proof}[Proof of \cref{thm:integer partition}]
	Let $G$ be a graph with $m(G) \leq k$ and let $E=E(G)$. Let $M_1$ be the graphic matroid on the ground set $E$ and let $M_2$ be the matroid given by \cref{lem:m2 is matroid}, whose independent sets are those $K \subseteq G$ with $m_2(K) \leq k$. We wish to prove that $E$ can be partitioned into an independent set from $M_1$ and an independent set from $M_2$; by \cref{thm:matroid partition}, it suffices to prove that $r_1(J) + r_2(J) \geq e_J$ for all $J \subseteq G$.

	So fix some $J \subseteq G$, and let its connected components be $J_1,\dots,J_t$. We then have that $r_1(J) = v_J-\omega_J = v_J-t$. As $m(G) \leq k$, we certainly have that $m(J_i) \leq k$ for all $i$, and hence \cref{lem:check-sc} implies that there exist $J'_i \subseteq J_i$ with $m_2(J'_i) \leq k$ and $e_{J_i} \leq e_{J'_i}+v_{J_i}-1$. Let $J' = J_1' \cup \dotsb \cup J_t'$. If $J'$ is a matching, then $m_2(J') \leq 1 \leq k$. If not, then its maximal $2$-density is attained on some connected component, hence $m_2(J') = \max_i m_2(J_i') \leq k$. Therefore, $J'$ is independent in $M_2$, which implies that
	\[
		r_2(J) \geq r_2(J') = e_{J'} =\sum_{i=1}^t e_{J_i'} \geq \sum_{i=1}^t (e_{J_i} - (v_{J_i}-1)) = e_J - (v_J-t).
	\]
	Recalling that $r_1(J) = v_J-t$, we conclude that $r_1(J) + r_2(J) \geq e_J$, as claimed.
\end{proof}

\end{document}